\documentclass[reqno,a4paper]{amsart}
\usepackage{amsmath,latexsym}
\usepackage[psamsfonts]{amssymb}
\usepackage{times}
\usepackage[mathcal]{euscript}

\usepackage{graphicx}
\graphicspath{ {images/} }

\numberwithin{equation}{section}
\usepackage{color}
\usepackage{array}

\usepackage{makecell}
\usepackage{curves}
\usepackage{tikz}
\usepackage{tikz-3dplot}

\numberwithin{equation}{section} \textwidth=140mm \textheight=200mm
\newcommand{\VVV}[1]{\left(\!\!\begin{array}{c}#1\end{array}\!\!\right)}

{\bf}{\it}
\newtheorem{theorem}{Theorem}[section]
\newtheorem{lemma}[theorem]{Lemma}
\newtheorem{corollary}[theorem]{Corollary}

\newtheorem{definition}[theorem]{Definition}
\newtheorem{proposition}[theorem]{Proposition}
\newtheorem{remark}[theorem]{Remark}

\begin{document}

\title[Threshold of  discrete Schr\"odinger operators]
{Threshold of  discrete Schr\"odinger operators with delta potentials
on $n$-dimensional lattice}

\author{Fumio Hiroshima\textsuperscript{1},
Zahriddin Muminov\textsuperscript{2}, Utkir
Kuljanov\textsuperscript{3}}

\address{\textsuperscript{1}Faculty of Mathematics, Kyushu University, Fukuoka, 819-0395,   Japan}
\email{hiroshima@math.kyushu-u.ac.jp}
\address{\textsuperscript{2}Faculty of Science and Technology, Nilai University, 71800, Nilai, Malaysia}
\thanks{CONTACT Z.~Muminov.
 Email: zimuminov@gmail.com}
 \email{zimuminov@gmail.com}
\address{\textsuperscript{3}Faculty of Mathematics, Samarkand State University,  703004,
Samarkand, Uzbekistan} \email{utkir\_nq83@mail.ru}



\begin{abstract}
Eigenvalue behaviors of Schr\"odinger operator defined on
$n$-dimensional lattice with $n+1$ delta potentials is studied. It
can be shown that lower threshold eigenvalue and lower threshold
resonance are appeared for $n\geq 2$, {and lower super-threshold
resonance appeared for $n=1$.}
\end{abstract}

 \maketitle

\section{Introduction}
Behavior of eigenvalues below the essential spectrum of standard
Schr\"odinger operators of the form $-\Delta +\varepsilon  V$
defined on $L^2({\mathbb R}^n)$
 is considerably studied so far. Here $V$ is a negative potential and $\varepsilon  \geq 0$ is a parameter which is varied.
When $\varepsilon $ approaches to some critical point $\varepsilon
_c\geq 0$, each negative eigenvalues approaches to the left edge of
the essential spectrum, and consequently they are absorbed into it.
A mathematical crucial problem is to specify whether a negative
eigenvalue  survives as an eigenvalue or a threshold resonance on
the edge of the essential spectrum at the critical point
$\varepsilon _c$. Their behaviors depend on the spacial dimension
$n$. Suppose that $V$ is relatively compact with respect to
$-\Delta$. Then the essential spectrum of $-\Delta+\varepsilon  V$
is $[0,\infty)$. Roughly speaking $-\Delta f+\varepsilon _c V f=0$
implies that $f=-\varepsilon _c (-\Delta)^{-1} Vf$ and
$$\|(-\Delta)^{-1} g\|_{L^2}^2=\int_{{\mathbb R}^n}|\hat g(k)|^2/|k|^4 dk,\quad
\|(-\Delta)^{-1} g\|_{L^1}=\int_{{\mathbb R}^n}|\hat g(k)|/|k|^2
dk,$$ where $g=Vf$. Hence it may be expected that $f\in L^2({\mathbb
R}^n)$ if $n\geq 5$ and $f\in L^1({\mathbb R}^n)$ for $n=3,4$. If
$0$ is an eigenvalue, it is called an embedded eigenvalue or
threshold eigenvalue. Hence it may be expected that an embedded
eigenvalue exists for $n\geq 5$. On the other hand for $n=3,4$, the
eigenvector is predicted to be in $L^1({\mathbb R}^n)$, and then $0$
is called a threshold resonance.

The discrete Schr\"odinger operators have attracted considerable
attentions for both combinatorial Laplacians and quantum graphs; for
some recent summaries refer to see \cite{C97, G01, BCFK06, EKKST08,
BK12, P12, KS13} and the references therein. Particularly,
eigenvalue behavior of  discrete Schr\"odinger operators  are
discussed in e.g. \cite{ALMM06,EKW10, BS12, HSSS12} and  are briefly
discussed in \cite{FIC,LB09, HSSS12} when potentials are  delta
functions with a single point mass. In \cite{ALMM06}  an explicit
example of a $-\Delta-V$  on the three-dimensional lattice
$\mathbb{Z}^3$, which possesses both a \textit{lower} threshold
resonance and a \textit{lower} threshold eigenvalue, is
constructed, where $-\Delta$ stands for the standard discrete
Laplacian in $\ell^2(\mathbb Z^n)$ and
 $V$
is a multiplication operator by the function
\begin{align}
\label{star} \hat V(x)=\mu\delta_{x0}+\frac{\lambda}{2}
\sum_{|s|=1}\delta_{xs}, \qquad \lambda\geq0,\mu\geq 0,
\end{align}
where $\delta_{xs}$ is the Kronecker delta.

The authors of  \cite{LB09} considered the restriction of this
operator to the Hilbert space $\ell^2_{\rm e}(\mathbb{Z}^3)$ of all
even functions in $\ell^2(\mathbb{Z}^3)$. They investigated the
dependence of the number of eigenvalues of $H_{\lambda\mu}$ on
$\lambda,\mu$ for $\lambda> 0,\mu > 0$, and they showed that all
eigenvalues arise either from a \textit{lower} threshold resonance
or from \textit{lower} threshold eigenvalues under a variation of
the interaction energy. Moreover, they also proved that the first
\textit{lower} eigenvalue  of the Hamiltonian $-\Delta- V$ arises
only from a \textit{lower}  threshold resonance under a variation of
the interaction energy. A continuous version,  two-particle
Schr\"odinger operator, is shown by Newton (see p.1353 in
\cite{N77}) and proved by Tamura \cite[Lemma 1.1]{Tam2} using a
result by Simon \cite{S}. In case $\lambda=0$, Hiroshima et.al.
\cite{HSSS12} showed that an threshold eigenvalue does appear for $n
\geq 5$ but does not for $1\leq n\leq 4$.

There are still  interesting spectral properties of the
  discrete  Schr\"odinger operators
with potential of the form \eqref{star}.

 In this paper, we investigate the spectrum of $H_{\lambda\mu}$,
specifically, {\it lower} and {\it upper} threshold eigenvalues and
threshold resonances for {\it any}
$$(\lambda,\mu)\in
{\mathbb R}^2\quad \mbox{ and } \quad n\geq 1.$$ {We emphasize that
there also appears so-called super-threshold resonances in our model
for $n=1$.} See Proposition \ref{takahashi}. The definitions of
these are given in Definition \ref{reson def even}. Our result is an
extensions of \cite{LB09,ALMM06,HSSS12}.

 In this paper,  we study,  in particular,
eigenvalues in $(-\infty, 0)$, lower threshold eigenvalues, lower
threshold resonances {and lower super-threshold resonances}. In a
similar manner to this, we can also investigate eigenvalues in
$(2n,\infty) $, upper threshold eigenvalues, upper threshold
resonances and {and upper super-threshold resonances}, but we left
them to readers, and we focus on studying the spectrum contained in
$(-\infty,0]$.

The paper is organized as follows. In Section 2,  a discrete
Schr\"odinger operator  in  the coordinate and momentum
representation is described, and it is decomposed into direct sum of
operators $H_{\lambda\mu}^{\mathrm{e}} $ and
$H_{\lambda}^{\mathrm{o}} $. The spectrum  of
$H_{\lambda\mu}^{\mathrm{e}} $ and $H_{\lambda}^{\mathrm{o}} $ are
investigated in Section 3. Section 4 is devoted to showing main
results,  Theorems \ref{Main 1} and \ref{Main 2}. The proofs of some
lemmas belong to Appendix.

\section{Discrete Schr\"odinger operators on lattice}
Let $\mathbb{Z}^{n}$ be the ${n}$--dimensional lattice, i.e. the
${n}$--dimensional integer set. The Hilbert space of  $\ell^2$
sequences on $\mathbb{Z}^n$ is denoted by $ \ell^2(\mathbb{Z}^n)$. A
notation $\mathbb{T} ^{n}= (\mathbb{R}/2\pi \mathbb{Z})^n
=(-\pi,\pi]^{n}$
 means the $n$-dimensional torus (the first Brillouin zone, i.e., the dual
group of $\mathbb{Z}^n$) equipped with its Haar measure, and let $
L^2_{\rm e}(\mathbb{T} ^n)$ (resp. $ L^2_{\rm o}(\mathbb{T} ^n)$)
denote the subspace of all
 even (resp. odd) functions of the Hilbert space $L^2(\mathbb{T} ^n)$   of $L^2$-functions on
 $\mathbb{T} ^n$. Let
$\langle \cdot,\cdot\rangle$ mean the inner product on
$L^2(\mathbb{T}^n)$.

Let $T(y)$ be the  shift operator by $y\in \mathbb{Z}^{n}$: $
(T(y)f)(x)=f(x+y)$ for $f\in\ell^2(\mathbb{Z}^n)$ and $x\in
\mathbb{Z}^n$. The standard discrete Laplacian $\Delta$ on
$\ell^2(\mathbb{Z}^n)$ is usually associated with the bounded
self-adjoint multidimensional
  Toeplitz-type
 operator:
\begin{align*}%
\Delta=\frac{1}{2}\sum_{x\in {\mathbb{Z}}^n \atop |x|=1}(T(x)-T(0)).
\end{align*}
Let us define the discrete Schr\"odinger operator on
 $\ell^2(\mathbb{Z}^n)$  by
 \begin{align*}%
    \hat H_{\lambda\mu}=-\Delta-\widehat{V},
\end{align*}
where the potential
 $\widehat{V}$ depends on two parameters $\lambda,\mu\in \mathbb{R}$ and
 satisfies
$$
(\widehat{V}f)(x)=\left\{
       \begin{array}{ll}
         \mu f(x), & \hbox{if}\quad  x=0 \\
         \frac{\lambda}{2} f(x), & \hbox{if}\quad |x|=1 \\
         0, & \hbox{if} \quad |x|>1
       \end{array}
     \right.
, \quad f\in \ell^2(\mathbb{Z}^n),\,x\in \mathbb{Z}^n,
$$
which awards $\hat H_{\lambda\mu}$ to be a bounded self-adjoint
operator. Let $\mathcal{F}$ be  the standard Fourier transform
    $\mathcal{F}:L^2(\mathbb{T} ^n) \longrightarrow   \ell^2(\mathbb{Z}^n)$ defined by
 $(Ff)(x)=\frac{1}{(2\pi)^n}\int_{{\mathbb T}^n} f(\theta) e^{ix\theta} d\theta$ for $f\in L^2({\mathbb T}^n)$ and $x\in {\mathbb Z}^n$.
 The inverse Fourier transform is then given by
 $(F^{-1}f)(\theta)=\sum_{x\in{\mathbb Z}^n} f(x) e^{-ix\theta}$ for $f\in \ell^2({\mathbb Z}^n)$ and $\theta\in{\mathbb T}^n$.  The Laplacian $\Delta$  in the momentum
representation
 is defined as
$$
\widehat{\Delta}=\mathcal{F}^{-1}\Delta \mathcal{F},$$ and
$\widehat{\Delta}$ acts as the multiplication operator:
$$
(\widehat{\Delta}\hat{f})(p) =-E(p)
 \hat{f}(p),$$
 where $E(p)$ is given by
 $$E(p)=\sum_{j=1}^n(1-\cos p_j).$$
 In the physical literature, the function
 $\sum_{j=1}^n(1-\cos p_j)$,
 being  a real valued-function on $\mathbb{T} ^n$,   is called the {\it
dispersion relation} of the Laplace operator. We also define the
discrete Schr\"odinger operator in momentum representation. Let
$H_0=-\hat{\Delta}$. The operator $H_{\lambda\mu}$,  in the momentum
representation, acts in the Hilbert space  $L^2(\mathbb{T} ^n) $ as
$$
H_{\lambda\mu}=H_0-V,
$$
where $V$ is an integral operator of convolution type
\begin{align*}
( Vf)(p)=(2\pi )^{-\frac{n}{2}}\int_{\mathbb{T} ^n}  v (p-s)f(s)d
s,\quad f\in L^2(\mathbb{T} ^n).
\end{align*}
Here the kernel function $ v(\cdot)$ is the Fourier transform of
$\widehat{V}(\cdot)$
 computed as
\begin{align*}%
 v(p)=\frac{1}{(2 \pi )^{\frac{n}{2}}} \left( \mu +
{\lambda} \sum_{i=1}^n\cos p_i \right),
 \end{align*}
and it allows  the potential operator $V$ to get the representation
$V=V_{\lambda\mu}^{\rm e}+V_{\lambda}^{\rm o}$, where
\begin{align*}
    V_{\lambda\mu}^{\rm e}=    \mu \langle
\cdot,\mathrm{c}_0\rangle\mathrm{c}_0+\frac{\lambda}{2} \sum_{j=1}^n
\langle \cdot,\mathrm{c}_j\rangle\mathrm{c}_j, \quad
V_{\lambda}^{\rm o}=\frac{\lambda}{2} \sum_{j=1}^n \langle
\cdot,\mathrm{s}_j\rangle\mathrm{s}_j.
\end{align*}
Here $\{\mathrm{c}_0,\mathrm{c}_j,\mathrm{s}_j:j=1,\dots,n\}$ is an
orthonormal system in $L^2(\mathbb{T}^n)$, where
\begin{align*}
{\mathrm{c}}_0(p)=\frac{1}{(2\pi)^{\frac{n}{2}}},\quad
{\mathrm{c}}_j(p)=\frac{\sqrt{2}}{(2\pi)^{\frac{n}{2}}}\cos
p_j,\quad
{\mathrm{s}}_j(p)=\frac{\sqrt{2}}{(2\pi)^{\frac{n}{2}}}\sin
p_j,\quad j=1,\dots,n.
\end{align*}
One can check easily that the subspaces $L_{\rm e}^2(\mathbb{T} ^n)$
of all even  functions and $L_{\rm o}^2(\mathbb{T} ^n)$ of all odd
functions in $L^2(\mathbb{T} ^n)$ reduce $H_{\lambda\mu}$. Adopting
$V=V_{\lambda\mu}^{\rm e}+V_{\lambda}^{\rm o}$, we can see that
the restriction $H_{\lambda\mu}^{\mathrm{e}} $ (resp.
$H_{\lambda}^{\rm o}$)  of the operator $H_{\lambda\mu}$ to $L_{\rm
e}^2(\mathbb{T} ^n)$ (resp. $L_{\rm o}^2(\mathbb{T} ^n)$)
 acts with the  form
\begin{align*}
H_{\lambda\mu}^{\mathrm{e}} =H_0-V_{\lambda\mu}^{\rm e} \quad ({\rm
resp.}\ H_{\lambda}^{\rm o}=H_0-V_{\lambda}^{\rm o}).
\end{align*}
Hence $H_{\lambda\mu} $ is decomposed into the even Hamiltonian and
the odd Hamiltonian:
$$
H_{\lambda\mu}= H_{\lambda\mu}^{\mathrm{e}} \oplus
H_{\lambda}^{\mathrm{o}}
$$
under the decomposition $L^2(\mathbb{T}^n)=L^2_{\rm
e}(\mathbb{T}^n)\oplus L_{\rm o}^2(\mathbb{T}^n)$. We have the
fundamental proposition below:
\begin{proposition}
It follows that
  $\sigma_{\text{\rm ess}}(H_{\lambda\mu})=\sigma_{\text{ac}}(H_{\lambda\mu})=[0,2n]$.
\end{proposition}
\begin{proof}
The perturbation $V$ is a finite rank operator and then the
essential spectrum of the operator $H_{\lambda\mu}$ fills in
$[0,2n]=\sigma_{\rm ess}(H_0)$.

Let $\mathcal{H}_{\rm ac}$ be the absolutely continuous part of
$H_{\lambda\mu}$. It can be seen that the wave operator
$W_\pm=s-\lim_{t\to\pm\infty} e^{itH_{\lambda\mu}}e^{-itH_0}$ exists
and is complete since $H_{\lambda\mu}$ is a finite rank perturbation
of $H_0$. This implies that $H_0$ and
 $H_{\lambda\mu}\lceil_{{\mathcal{H}}_{\rm ac}} $ are unitarily equivalent by $W_\pm^{-1} H_0 W_\pm=
H_{\lambda\mu}\lceil_{\mathcal{H}_{\rm ac}}$. Then $\sigma_{\rm
ac}(H_0)= \sigma_{\rm ac}(H_{\lambda\mu})=[0,2n]$.
\end{proof}

In what follows, we shall study the spectrum of $H_{\lambda\mu}$ by
investigating the spectrum of $H_{\lambda\mu}^{\mathrm{e}} $ and $
H_{\lambda}^{\mathrm{o}} $ separately.

\section{Spectrum  of $H_{\lambda\mu}^{\mathrm{e}} $}
\subsection{Birman-Schwinger principle for $z\in {\mathbb C}\setminus[0,2n]$}
The Birman-Schwinger principle helps us to reduce the problem to the
study of spectrum of a finite dimensional linear operator: a matrix.

We denote the resolvent of Laplacian $H_0$ by  $(H_0-z)^{-1}$, where
$z\in \mathbb{C}\setminus [0,2n]$. We can see that
$(H_0-z)^{-1}V_{\lambda\mu}^{\rm e}$ is a finite rank operator. Let
$M_{n+1}$ denote the linear hull of $\{c_0,\cdots,c_n\}$. Then
$M_{n+1}$ is an $(n+1)$-dimensional subspace of $L_{\rm
e}^2(\mathbb{T} ^n)$. Furthermore we define $\tilde
M_{n+1}=(H_0-z)^{-1}M_{n+1}$ for $z\in \mathbb{C}\setminus [0,2n]$.
Then $\tilde M_{n+1}$ is also an $(n+1)$-dimensional subspace of
$L_{\rm e}^2(\mathbb{T} ^n)$ since $(H_0-z)^{-1}$ is invertible. We
define $C_1:\mathbb{C}^{n+1}\to L_{\rm e}^2(\mathbb{T} ^n)$ by the
map
$$C_1: \mathbb{C}^{n+1}\ni \VVV {w_0\\ \vdots\\ w_n}
\mapsto (H_0-z)^{-1}\left(\mu w_0 c_0+ \frac{\lambda}{2}\sum_{j=1}^n
w_j c_j\right)\in \tilde M_{n+1},$$ and define $C_2:L_{\rm
e}^2(\mathbb{T} ^n)\to \mathbb{C}^{n+1}$ by the map
$$C_2:L_{\rm e}^2(\mathbb{T} ^n)\ni \phi\mapsto
\VVV{ \langle\phi,c_0\rangle\\ \vdots\\
\langle \phi,c_n\rangle}\in \mathbb{C}^{n+1}.$$
Then we have the sequence of maps:
\begin{align}
\label{seq} L_{\rm e}^2(\mathbb{T} ^n)
 \stackrel{C_2}{\longrightarrow}
\mathbb{C}^{n+1} \stackrel{C_1}{\longrightarrow} L_{\rm
e}^2(\mathbb{T} ^n)
\end{align}
and $C_1C_2:L_{\rm e}^2(\mathbb{T} ^n)\to L_{\rm e}^2(\mathbb{T}
^n)$.
Notice that $C_1$ and $C_2$ depend on the choice of $z$. We directly
have
\begin{align}\label{G1e}
   (H_0-z)^{-1}V_{\lambda\mu}^{\rm e}=C_1C_2.
   \end{align}
Define
$$
G_{\rm e}(z)=C_2C_1:\mathbb{C}^{n+1}\to \mathbb{C}^{n+1}.
$$
We shall show the explicit form of $G_{\rm e}(z)$ in \eqref{G=AB}
below.
\begin{lemma}[Birman-Schwinger principle for {$z\in {\mathbb C}\setminus[0,2n]$}]
\label{lem 1e}\
\begin{enumerate}
\item[(a)] 
$z\in \mathbb{C}\setminus [0,2n]$ is an eigenvalue of
$H_{\lambda\mu}^{\mathrm{e}} $ if and only if  $1\in \sigma(
 G_{\rm e} (z))$.

\item[(b)] Suppose that  $z\in \mathbb{C}\setminus [0,2n]$ and
$(\lambda,\mu)$ satisfies ${\rm det}(G_{\rm e}(z)-{\rm I})=0$. Then
the vector $Z=\VVV{w_0\\ \vdots\\ w_n}\in \mathbb{C}^{n+1}$ is an
eigenvector of $G_{\rm e}(z)$ associated with eigenvalue $1$ if and
only if
 $ f=C_1Z$, i.e.
\begin{align}\label{eigen fe}
f(p)=\frac{1}{(2\pi) } \frac{1}{E(p)-z}\left( \mu w_0+
\frac{\lambda}{\sqrt {2}} \sum_{j=1}^n w_j\cos p_j \right)
\end{align}
 is an eigenfunction of $H_{\lambda\mu}^{\mathrm{e}} $ associated with eigenvalue $z$.
\end{enumerate}\end{lemma}
\begin{proof}
It can be seen that
    $H_{\lambda\mu}^{\mathrm{e}} f=zf$ if and only if
    $
     f=
    (H_0-z)^{-1}V_{\lambda\mu}^{\rm e}f$.
Then $z\in \mathbb{C}\setminus [0,2n]$ is an eigenvalue of
$H_{\lambda\mu}^{\mathrm{e}} $ if and only if  $1\in \sigma(
 (H_0-z)^{-1}V_{\lambda\mu}^{\rm e})$.
 Hence
 $z\in \mathbb{C}\setminus
[0,2n]$ is an eigenvalue of $H_{\lambda\mu}^{\mathrm{e}} $ if and
only if  $1\in \sigma(C_1C_2)$ by
 the fact $\sigma(C_1C_2)\setminus\{0\}=\sigma(C_2C_1)\setminus\{0\}$.

Then it  completes the proof of (a). We can also see that
$C_2C_1Z=Z$ if and only if
 $ f=(H_0-z)^{-1}V_{\lambda\mu}^{\rm e}f=C_1C_2f$, where
$f=C_1Z$. Then the function  $f$ coincides with \eqref{eigen fe}.
\end{proof}

\subsection{Birman-Schwinger principle for $z=0$}
We consider the Birman-Schwinger principle for $z=0$, which is the
edge of the continuous spectrum of $H_{\lambda\mu}^{\mathrm{e}}$,
and it is the main issue to specify whether it is eigenvalue or
threshold of $H_{\lambda\mu}^{\mathrm{e}} $.

 In order to discuss
$z=0$ we extend the eigenvalue equation $H_{\lambda\mu}^{\mathrm{e}}
f=0$ in $L^2_{\rm e}(\mathbb{T}^n)$ to that in $L_{\rm
e}^1(\mathbb{T}^n)$. Note that $L_{\rm e}^2(\mathbb{T}^n)\subset
L_{\rm e}^1(\mathbb{T}^n)$. We consider the equation
\begin{equation}\label{Hf=0}
 E(p) f(p)-\frac{\mu}{(2\pi)^n}
 \int_{\mathbb{T}^n}f(p)dp-\frac{\lambda}{(2\pi)^n}
 \sum_{j=1}^n
 \cos p_j \int_{\mathbb{T}^n}\cos p_j f(p)dp =0
\end{equation}
in the Banach space $L^1_{\rm e}(\mathbb{T}^n)$. Conveniently, we
describe (\ref{Hf=0}) as $H_{\lambda\mu}^{\mathrm{e}} f=0$. Since we
consider a solution $f\in L_{\rm e}^1(\mathbb{T}^n)$, the integrals
 $\int_{\mathbb{T}^n}f(p)dp$ and
 $\int_{\mathbb{T}^n}\cos p_j f(p)dp$ are  finite for $j=1,...,n$.

The unique singular point of $1/E(p)$ is $p=0$, and  in the
neighborhood of $p=0$, we have $E(p)\approx |p|^2$. Then  the
following lemma is fundamental, and its proof is straightforward.

\begin{lemma}\label{Incl L2}
Let $h(p)={\varphi(p)}/{E(p)}$, where $\varphi\in C(\mathbb{T}^n)$.
Then (a)-(e) follow.
\begin{itemize}
\item[(a)]  It follows that $h\in L^2(\mathbb{T}^n)$ for $n\geq 5$, and
$h\in L^1(\mathbb{T}^n)$ for $n\geq 3$.
\item[(b)] Let $1\leq n\leq 4$ and $h\in L^2(\mathbb{T}^n)$. Then
$ \varphi(0)=0$.
\item[(c)] Let $1\leq n\leq 4$,
 $|\varphi(p)|<C |p|^{\alpha_n}$ for some $C>0$ and $\alpha_n>\frac{4-n}{2}$.
 Then
 $ h\in L^2(\mathbb{T}^n)$.
\item[(d)] Let $n=1,2$ and $h\in L^1(\mathbb{T}^n)$. Then
$ \varphi(0)=0$.
\item[(e)]
Let $n=1,2$,
 $|\varphi(p)|<C |p|^{\alpha_n}$ for some $C>0$ and $\alpha_n>2-n$.
 Then
$h\in L^1(\mathbb{T}^n)$.
\end{itemize}
\end{lemma}
Operator  $H_0^{-1}$ is not bounded  in $L^2_{\rm e}(\mathbb{T}^n)$
as well as in $L^1_{\rm e}(\mathbb{T}^n)$. It is however obvious by
Lemma \ref{Incl L2} and $V_{\lambda\mu}^{\rm e}f\in C(\mathbb{T}^n)$
that
\begin{align}
L^2_{\rm e}(\mathbb{T}^n)\ni f\mapsto     H_0^{-1}V_{\lambda\mu}^{\rm e}f \in L^2_{\rm e}(\mathbb{T}^n),&\quad  n\geq 5,\\
L^1_{\rm e}(\mathbb{T}^n)\ni f\mapsto   H_0^{-1}V_{\lambda\mu}^{\rm
e}f \in L^1_{\rm e}(\mathbb{T}^n) ,& \quad n\geq 3.
\end{align}
Thus for $n\geq 3$ we can extend operators $C_1$ and $C_2$ defined
in the previous section. Let  $n\geq3$ and $Z=\VVV{ w_0\\
\vdots\\w_n}$. $\bar C_1: \mathbb{C}^{n+1} \to L^1_{\rm
e}(\mathbb{T} ^n)$ is  defined by
$$\bar C_1 Z=
\frac{1}{(2\pi) } \frac{1}{E(p)}\left( \mu w_0+ \frac{\lambda}{\sqrt
{2}} \sum_{j=1}^n w_j\cos p_j \right)$$ and $\bar C_2:L_{\rm
e}^1(\mathbb{T} ^n)\to \mathbb{C}^{n+1}$  by
$$\bar C_2:L_{\rm e}^1(\mathbb{T} ^n)\ni \phi\mapsto \VVV{\int_{\mathbb{T}^n} \phi(p) c_0dp\\
\int_{\mathbb{T}^n} \phi(p) c_1(p) dp\\
\vdots\\
\int_{\mathbb{T}^n} \phi(p) c_n(p) dp }\in \mathbb{C}^{n+1}.$$ Then
$\overline{C}_1\overline{C}_2: L^1_{\rm e}(\mathbb{T} ^n)\to
L^1_{\rm e}(\mathbb{T} ^n) $. Consequently $ G_{\rm
e}(0)=\overline{C}_2\overline{C}_1:\mathbb{C}^{n+1}\to
\mathbb{C}^{n+1} $ is described as an  $(n+1)\times (n+1)$ matrix.
\begin{lemma}
Let $n\geq 3$. Then it follows that (1) $\displaystyle \lim_{z\to 0}
G_{\rm e}(z)=G_{\rm e}(0)$ and (2)
$\sigma(H_0^{-1}V_{\lambda\mu}^{\rm e}) \setminus\{0\}=\sigma(G_{\rm
e}(0))\setminus~\{0\} $.
\end{lemma}
\begin{proof}
The proof is straightforward.
\end{proof}

\begin{lemma}[Birman-Schwinger principle for $z=0$]
\label{lem 1e res} Let $n\geq3$. Then (a) and (b) follow. (a)
Equation $H_{\lambda\mu}^{\mathrm{e}} f=0$  has a solution in
$L^1(\mathbb{T}^n)$ if and only if $1\in \sigma(G_{\rm e} (0))$. (b)
Let ${Z}=\VVV{ w_0\\  \vdots\\w_n}\in \mathbb{C}^{n+1}$ be the
solution of $G_{\rm e}(0)Z=Z$ if and only if
\begin{equation}\label{eigen fe res}
f(p)=\overline{C}_1{Z}(p)=\frac{1}{(2\pi)^{\frac{n}{2}}} \frac{1}{E
(p)}\left( \mu w_0+ \frac{\lambda}{{\sqrt{2}}} \sum_{j=1}^n w_j\cos
p_j \right)
\end{equation}
is a solution of $H_{\lambda\mu}^{\mathrm{e}} f=0$, where
$w_0,\cdots, w_n$ are actually described by
\begin{align}\label{w}
w_0= \frac{1}{(2\pi)^{\frac{n}{2}}} \int_{\mathbb{T}^n}f(p)dp, \quad
w_j=\frac{\sqrt 2}{(2\pi)^{\frac{n}{2}}} \int_{\mathbb{T}^n}f(p)\cos
p_jdp,\quad j=1,\dots,n.
\end{align}
\end{lemma}
\begin{proof}
Let us consider $H_{\lambda\mu}^{\mathrm{e}} f=0$ in $L^1_{\rm
e}(\mathbb{T}^n)$. Hence $f=H_0^{-1} V_{\lambda\mu}^{\rm e}f$ in
$L^1_{\rm e}(\mathbb{T}^n)$. Then $L^1$-solution of
$H_{\lambda\mu}^{\mathrm{e}} f=0$ exists if and only if  $1\in
\sigma(H_0^{-1}V_{\lambda\mu}^{\rm e})$, and hence $L^1$-solution of
$H_{\lambda\mu}^{\mathrm{e}} f=0$ exists if and only if  $1\in
\sigma(\bar  C_1\bar C_2)$. Due to the fact $\sigma(\bar C_1\bar
C_2)\setminus\{0\}= \sigma(G_{\rm e}^0)\setminus\{0\}$ the proof of
(a) is complete. We can also see that   $\bar C_2\bar C_1Z=Z$ if and
only if
 $ f=H_0^{-1}V_{\lambda\mu}^{\rm e}f=\bar C_1\bar C_2f$, where
$f=\bar C_1Z$. Then the function  $f$ coincides with \eqref{eigen fe
res}. This fact   ends the proof of (b).
\end{proof}

\subsection{Zeros of $\det (G_{\rm e}(z)-{\rm I})$}
\label{delta c zero}
\subsubsection{Factorization}
By the Birman-Schwinger principle  in what follows we focus on
investigating the spectrum of the  $(n+1)\times (n+1)$-matrix
$G_{\rm e}(z)$. Since $G_{\rm e}(z)$ is defined for $z\in (-\infty,
0)$ for $n=1,2$, and $z\in (-\infty, 0]$ for $n\geq3$. Hence in this
section we suppose that $
z\in\left\{\begin{array}{ll} (-\infty, 0)& n=1,2,\\
(-\infty,0]& n\geq3.
\end{array}
\right. $ As  the function $E(p)=E(p_1,\dots,p_n)$ is invariant with
respect to the permutations of  its arguments $p_1,\dots,p_n$, the
integrals used for studying the spectrum of $G_{\rm e}(z)$:
  \begin{align}
a(z)&=  \langle c_0,(H_0-z)^{-1} c_0 \rangle =\frac{1}{(2\pi)^n}
\int_{\mathbb{T} ^n} \frac{1}{E(p)-z}dp,  \\
b(z)&=\frac{1}{\sqrt{2}}\langle c_0,(H_0-z)^{-1} c_j \rangle =
\frac{1}{(2\pi)^n} \int_{\mathbb{T} ^n} \frac{\cos p_j }{E(p)-z}dp, \\
c (z)&=\frac{1}{2}\langle c_j,(H_0-z)^{-1} c_j \rangle =
\frac{1}{(2\pi)^{n}} \int_{\mathbb{T} ^n} \frac{\cos^2 p_j
 }{E(p)-z}dp,
\\
d(z)&=\frac{1}{2}\langle c_i,(H_0-z)^{-1} c_j \rangle
=\frac{1}{(2\pi)^{n}} \int_{\mathbb{T} ^n} \frac{\cos p_i\cos
p_j  }{E(p)-z}dp,\quad i\not=j, \\
\label{s(z)}
  s(z)&=\frac{1}{2}\langle \mathrm{s}_j,(H_0-z)^{-1} \mathrm{s}_j \rangle=
\frac{1}{(2\pi)^{n}} \int_{\mathbb{T}^n} \frac{\sin^2 p_j
}{E(p)-z}dp.
\end{align}
also do not depend on the particular choice of indices $0\leq
i,j\leq n$. Note that $a(z),b(z), c(z)$ and $s(z)$ are defined for
$n\geq1$ but $d(z)$  for  $n\geq 2$. From the definition of $G_{\rm
e}(z)=(a_{ij})_{0\leq i,j\leq n}$, coefficients  $a_{ij}=a_{ij}(z)$
are explicitly described as
\begin{align*}
\left\{\begin{array}{ll} a_{00}(z)=\mu a(z),\quad
a_{0j}(z)=\frac{\lambda}{\sqrt{2}}b(z),
& j=1,\dots, n,\\
a_{i0}(z)=\sqrt{2}\mu b(z),a_{ii}(z)=\lambda c(z),&i=1,...,n\\
a_{ij}(z)=\lambda d(z), &  i,j=1,\dots n,\,j\neq i,
\end{array}
\right.
\end{align*}
Hence for $n\geq2$ the matrix $G_{\rm e}(z)$ has the form
\begin{align}\label{G=AB}
  G_{\rm e}(z)=\left(
           \begin{array}{ccccc}
             \mu a(z) &  \frac{\lambda}{\sqrt{2}} b(z) &  \ldots &  \ldots & \frac{\lambda}{\sqrt{2}} b(z) \\
             \sqrt{2}\mu b(z) & \lambda c(z) & \lambda d(z) &  \dots &  \lambda d(z) \\
             \vdots & \lambda d(z) & \ddots & \dots & \vdots\\
             \vdots &  \vdots & \dots & \ddots & \lambda d(z)\\
             \sqrt{2} \mu b(z) & \lambda d(z) &  \dots & \lambda d(z) & \lambda c(z)\\
           \end{array}
         \right)
\end{align}
and for $n=1$,
\begin{align}\label{G=AB1}
  G_{\rm e}(z)=\left(
           \begin{array}{cc}
             \mu a(z) &  \frac{\lambda}{\sqrt{2}} b(z)\\
             \sqrt{2}\mu b(z) & \lambda c(z)
              \end{array}
         \right).
\end{align}
In order to study the eigenvalue $1$ of $G_{\rm e}(z)$ we calculate
the  determinant of $G_{\rm e}(z)-{\rm I}$.
 \begin{lemma}\label{FF}
We have
\begin{align}\label{factorization}
\mathrm{det}(G_{\rm e}(z)-{\rm I})=
\delta_r(\lambda,\mu; z) 
\delta_c(\lambda;z),
\end{align}
where
\begin{align}
\label{sora} &\delta_r(\lambda,\mu;z)= \left\{  \begin{array}{ll}
 \big(1-\mu a(z)\big)\Big\{1-\lambda\big(c(z)+(n-1)d(z)\big)\Big\}-n  \lambda\mu
b^2(z),&n\geq2\\
 \big(1-\mu a(z)\big)(1-\lambda c(z))-\lambda\mu
b^2(z),&n=1,
\end{array}\right.\\
& \delta_c(\lambda;z)= \left\{
\begin{array}{ll}
\big\{\lambda(c(z)-d(z))-1\big\}^{n-1},&n\geq 2\\
1,&n=1.
\end{array}
\right.
\end{align}
\end{lemma}
\begin{proof}
It is a straightforward computation.
\end{proof}
By the factorization \eqref{factorization} we shall study zeros of $
\delta_r(\lambda,\mu; z)$ and $\delta_c(\lambda;z)$ separately to
see eigenvalues and resonances of $H_{\lambda\mu}^{\rm e}$. To see
this we introduce algebraic relations used to estimate zeros of
$\delta_r(\lambda,\mu; z)$ and $\delta_c(\lambda;z)$. Below  we
shall show the list of formulas of coefficients $a(z)$, $b(z)$, etc.
We set
\begin{align}\label{Notation}
\alpha(z)&=\left\{\begin{array}{ll}
 c(z)+(n-1)d(z),&n\geq 2\\
 c(z), & n=1,
 \end{array}
 \right.\\
 \gamma(z)&=  a(z)\alpha(z)-nb^2(z).
\end{align}

\begin{lemma}\label{alpha - b} For any $z<0$, the relations below hold:
\begin{align*}
&a(z)-b(z)=\frac{1}{n}+\frac{z}{n}a(z),\\
&\alpha(z)=(n-z)b(z),\\
&\gamma(z)= b(z).
\end{align*}
\end{lemma}
\begin{proof}
We directly see that
\begin{align*}
&a(z)-b(z)=\frac{1}{(2\pi)^n} \int_{\mathbb{T} ^n} \frac{(1-\cos
p_1)}{E(p)-z} dp=\frac{1}{n} \frac{1}{(2\pi)^n} \int_{\mathbb{T} ^n}
\frac{
\sum_{j=1}^n (1-\cos p_j) }{E(p)-z}dp\\
&
=\frac{1}{n} \frac{1}{(2\pi)^n} \int_{\mathbb{T} ^n}dp +\frac{z}{n}
\int_{\mathbb{T} ^n} \frac{1 }{E(p)-z}dp
=\frac{1}{n}+\frac{z}{n}a(z),
\end{align*}
and
\begin{align*}
   \alpha(z)
=\frac{1}{(2\pi)^{n}}\int_{\mathbb{T} ^n} \frac{\cos p_1(z-E(p))
 }{E(p)-z}dp +\frac{n-z}{(2\pi)^{n}} \int_{\mathbb{T} ^n} \frac{\cos
p_1  }{E(p)-z}dp=(n-z)b(z).
\end{align*}
From these  we can also get  the third equality of the lemma.
\end{proof}

\begin{lemma}\label{LemMon}
Functions $a(z)$,
 $\alpha(z)$,  $\gamma(z)$,  $b(z)$,  $c(z)-d(z)$ and
 $s(z)$ are monotonously increasing and positive  in
 $(-\infty,0]$.
 Moreover, their limits tend to zero as $z$ tends to $-\infty$.
\end{lemma}
\begin{proof}
We have the representations of $a(z)$, $\gamma(z)=b(z)$ and $s(z)$
by their definitions as
\begin{align}
& \alpha(z)=  \frac{1}{n}\frac{1}{(2\pi)^{n}} \int_{\mathbb{T}^n}
\frac{ \big ( \sum_{i=1}^n  \cos p_i \big)^2
}{E(p)-z}dp, \\
 &\gamma(z)=b(z)= \frac{1}{2n(2\pi)^{2n}}\int_{\mathbb{T}^n\times \mathbb{T}^n} \frac{ \big ( \sum_{j=1}^n ( \cos p_j-\cos q_j    ) \big)^2 }{(E(p)-z)(E(q)-z)}dpdq,\\
&c(z)-d(z)= \frac{1}{2}\frac{1}{(2\pi)^{n}} \int_{\mathbb{T}^n}
\frac{ \big (
 \cos p_1-\cos p_2 )^2 }{E(p)-z}dp, \\
& \label{sz}s(z)=  \frac{1}{(2\pi)^{n}} \int_{\mathbb{T}^n}  \frac{
\sin^2p_1}{E(p)-z}dp .
\end{align}
Indeed, for any fixed $p\in \mathbb{T}^n$, all the integrands 
are positive and  monotonously increasing as functions of $z$, and
we complete the proof.
\end{proof}
Note that in Lemma \ref{LemMon} $c(z)-d(z)$ is considered only in
the case of $n\geq2$.
\begin{lemma}\label{as<b}
The following relations hold:
\begin{align}
a(z)s(z)=b(z), & \quad \mbox{} n=1,\, z< 0,\nonumber\\
a(z)s(z)<b(z), & \quad \mbox{} n=2,\, z< 0,\nonumber\\
a(z)s(z)<b(z), & \quad \mbox{} n\geq 3, \, z\leq0,\nonumber\\
\label{s=1} c(z)-d(z)<s(z),&\quad   n\geq 2,\,z\leq 0.
\end{align}
\end{lemma}
\begin{proof}
  See Appendix \ref{ineq lem}.
\end{proof}

\begin{lemma}\label{lem asymp}
The function $ {a(z)}/{b(z)} $ is monotonously decreasing in
$(-\infty,0]$, and  there exist   limits:
\begin{align}
&\label{lim 1} \lim_{z\to -\infty}\frac{a(z)}{b(z)}=+\infty,\\
&\label{lim 2} \lim_{z\to  0-}\frac{a(z)}{b(z)}= \left\{
                                             \begin{array}{ll}
                                               1, &  n=1,2, \\
                                               \frac{a(0)}{b(0)}, &  n\geq 3.
                                                                                             \end{array}
                                           \right.
\end{align}
\end{lemma}
\begin{proof}
  See Appendix \ref{app lem assymp}.
\end{proof}
\subsubsection{Zeros of $\delta_r(\lambda,\mu;z)$}
We extend $\delta_r(\lambda,\mu;\cdot)$ and
$\delta_c(\lambda;\cdot)$, and discuss zeros of them to specify the
eigenvalue of $H_{\lambda\mu}^{\mathrm{e}} $. Let $z\in (-\infty,
0)$. Applying notation in \eqref{Notation}, we describe
$\delta_r(\lambda,\mu;z)$ as
\begin{align}
  \delta_r(\lambda,\mu;z)=\gamma(z)\mathcal{H}_z(\lambda,\mu) \label{del Hyper}
\end{align}
where
\begin{align}
  \mathcal{H}_z(\lambda,\mu)=
\Big(\lambda-\frac{a(z)}{\gamma(z)}\Big)
\Big(\mu-\frac{\alpha(z)}{\gamma(z)}\Big)-\frac{a(z)\alpha(z)-\gamma(z)}{\gamma^2(z)}.
\end{align}
or by Lemma \ref{alpha - b}, we have 
 \begin{align}\label{hyper}
    \mathcal{H}_z(\lambda,\mu)=\Big(\lambda-\frac{a(z)}{b(z)}\Big)
\Big(\mu-(n-z)\Big)-n.
 \end{align}
Instead of the equation $\delta_r(\lambda,\mu;z)=0$, the relation
\eqref{del Hyper}   allows us   to study the family of rectangular
hyperbola ${\mathfrak{H}_z}$ indexed by
$z\in\left\{\begin{array}{ll} (-\infty, 0)& n=1,2\\
(-\infty,0]& n\geq3
\end{array}
\right.$. i.e. equilateral  hyperbola ${\mathfrak{H}_z}$ on
$(\lambda,\mu)$-plane,  which is defined by
 $$
{\mathfrak{H}_z}=\{(\lambda,\mu)\in {\mathbb R}^2 |
 \mathcal{H}_z(\lambda,\mu)=0\}
 $$
with asymptote
$$
(\lambda_{\infty}(z),\mu_{\infty}(z))=(\frac{a(z)}{b(z)}, n-z).$$
Lemma \ref{lem asymp} implies  that  $\mathcal{H}_z(\lambda,\mu)$
can be extended to $z\in(-\infty,0]$ for any dimension $n\geq1$ as
\begin{align}\label{f1}
\bar{\mathcal{H}}_z(\lambda,\mu)=
\left\{\begin{array}{ll}\mathcal{H}_z(\lambda,\mu), &z<0,\\
(\lambda-X)(\mu-n)-n, &z=0.
\end{array}
\right.
\end{align}
Here $X=1$ for $n=1,2$ and $X=a(0)/b(0)$ for $n\geq 3$. Note that
$${\overline{\mathcal{H}}}_z(0,0)=\frac{a(z)}{b(z)}(n-z)-n=\frac{1}{b(z)}>0
$$ for $z<0$.
We also extend the family of hyperbola ${\mathfrak{H}_z}$, $z\in
(-\infty,0)$, to that of hyperbola $\overline{\mathfrak{H}}_z$
indexed by $z\in (-\infty,0]$ by
 $$
{\overline{\mathfrak{H}}_z}=\{(\lambda,\mu)\in \mathbb{R} \times
\mathbb{R}  |
 {\overline{\mathcal{H}}}_z(\lambda,\mu)=0\}.
 $$
By \eqref{f1} we see that $(\lambda,\mu)\in
{\overline{\mathfrak{H}}_0}$ satisfies the algebraic relation:
\begin{align}\label{f2}
(\lambda-X)(\mu-n)-n=0.
\end{align}
For any $z_1<z_2,$ $z_1,z_2\in (-\infty,0]$, we note that
 the hyperbola $\overline{\mathfrak H}_{z_1}$
{can be moved to  $\overline{\mathfrak H}_{z_2}$ in parallel} by the
vector
$g =\left(\begin{array}{l}\lambda_{\infty}(z_2)-\lambda_{\infty}(z_1)\\
\mu_{\infty}(z_2)-\mu_{\infty}(z_1)\end{array} \right) $ whose
components are positive. See Figure~\ref{f3}.
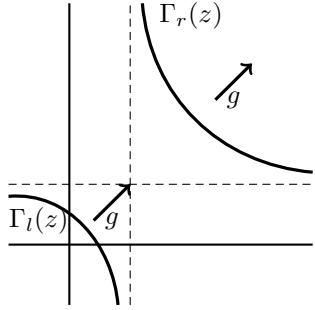
\begin{figure}[t]
\begin{center}
\begin{tikzpicture}[scale=0.8]
\useasboundingbox(-1,-2) rectangle(4,4);
\node[below]at(1.5,-1){$ $};
\draw[thick](-1,0)--(4,0); \draw[thick](0,-1)--(0,4);

\draw[densely dashed](1,-1)--(1,4); \draw[densely
dashed](-1,1)--(4,1);

\draw[very thick](-1,0.8)to[out=50,in=140,relative](0.8,-1);
\draw[very thick](1.2,4)to[out=320,in=220,relative](4,1.2);

\draw[->,very thick](0.4,0.4)--(1,1); \draw[->,very
thick](2.4,2.4)--(3,3);

\node at(2.7,2.4){$g$}; \node at(0.7,0.4){$g$}; \node
at(-0.5,0.4){$\Gamma_l(z)$}; \node at(2.0,3.8){$\Gamma_r(z)$};
\end{tikzpicture}
\end{center}
\vspace{-1.4cm}\caption{Hyperbola  moves as $z$ approaches to
$-\infty$ from $0$.} \label{f3}
\end{figure}
Let  $\Gamma_l(z)$ (resp. $\Gamma_r(z)$)  denote the left brunch
(resp. the right brunch) of the hyperbola
$\overline{\mathfrak{H}}_z$, i.e.
$$\overline{\mathfrak{H}}_z=
\Gamma_l(z)\cup \Gamma_r(z)\quad \mbox{ and }\quad  \Gamma_l(z)\cap
\Gamma_r(z)=\emptyset.$$ We then see that
 for any $z_2<z_1\leq 0$
 it follows that
\begin{align}\label{sch}
  \Gamma_l(z_1) \cap \Gamma_l(z_2)=\emptyset, \quad
  \Gamma_r(z_1) \cap \Gamma_r(z_2)=\emptyset .
  \end{align}
Let us see the behavior of $\delta_r(\lambda,\mu;z)$ near $z=0$ for
$n=1,2$.
\begin{lemma}\label{lak}
It follows that
\begin{align}\label{a(z) rep}
      a(z)&=\frac{1}{\sqrt{-z} \sqrt{2-z} },\quad n=1,\\
a(z)&=-\frac{\sqrt{2}}{2\pi}\ln(-z)+(\frac{1}{2}-\frac{\sqrt{2}}{\pi})+O(-z),\quad
\mbox{as}\quad z\to 0-,\quad n=2.
\end{align}
\end{lemma}
\begin{proof}
The proof  of this lemma can be found  in \cite{LakTil}.
\end{proof}
From this lemma we can see the behaviours of
$\delta_r(\lambda,\mu;z)$ as $z\to 0-$.
\begin{corollary}\label{delta1 con}
It follows that
\begin{align*}
\begin{array}{lll}
(n=1) & \lim_{z\to 0-}\delta_r(\lambda,\mu;z) &= \left\{
\begin{array}{ll}
\infty&(\lambda,\mu)\not\in \overline{\mathfrak H}_{0}\\
1-\mu&(\lambda,\mu)\in \overline{\mathfrak H}_{0}
\end{array}
\right.,\\
& & \\
(n=2) & \lim_{z\to 0-}\delta_r(\lambda,\mu;z) &=
\left\{\begin{array}{ll}
\infty&(\lambda,\mu)\not\in \overline{\mathfrak H}_{0}\\
1-\mu/2&(\lambda,\mu)\in \overline{\mathfrak H}_{0}
\end{array}\right.,\\
& & \\
(n\geq 3) & \lim_{z\to 0-}\delta_r(\lambda,\mu;z) &=
b(0){\overline{\mathcal{H}}}_0(\lambda,\mu).
\end{array}
\end{align*}
\end{corollary}
\begin{proof}
In the case of $n\geq3$ it is trivial to see that $\lim_{z\to
0-}\delta_r(\lambda,\mu;z) =
b(0){\overline{\mathcal{H}}}_0(\lambda,\mu)$. Then we consider cases
of $n=1,2$. We recall that
\begin{align*}
\delta_r(\lambda,\mu;z)=
\gamma(z)\overline{\mathcal{H}}_{z}(\lambda,\mu) =
\gamma(z)\overline{\mathcal{H}}_{0}(\lambda,\mu) +
\gamma(z)(\overline{\mathcal{H}}_{z}(\lambda,\mu)-\overline{\mathcal{H}}_{0}(\lambda,\mu))
\end{align*}
and
$$\gamma(z)=b(z)=a(z)-\frac{1}{n}-\frac{1}{n}za(z).$$
We can also directly see that for $n=1,2$
\begin{align*}
\overline{\mathcal{H}}_{z}(\lambda,\mu)
-\overline{\mathcal{H}}_{0}(\lambda,\mu)
&=\frac{1}{b(z)}(b(z)-a(z))(\mu-n)+z(\lambda-\frac{a(z)}{b(z)})\\
&= -\frac{1}{nb(z)}(1+za(z))(\mu-n)+z(\lambda-\frac{a(z)}{b(z)}).
\end{align*}
Together with them we have
\begin{align*}
\delta_r(\lambda,\mu;z)= (a(z)-\frac{1+za(z)}{n})
\overline{\mathcal{H}}_{0}(\lambda,\mu)+
\frac{a(z)}{b(z)}(1+za(z))(\frac{n-\mu}{n})+\xi,
\end{align*}
where
$$\xi=-za(z)(\lambda-\frac{a(z)}{b(z)})+\frac{1+za(z)}{n}
\left(\frac{(1+za(z))(\mu-n)}{nb(z)}+z(\lambda-\frac{a(z)}{b(z)})\right).$$
By Lemmas \ref{lak} and \ref{lem asymp} it is crucial to see that
$$\lim_{z\to 0-}za(z)=0,\quad
\lim_{z\to 0-}b(z)=\infty,\quad \lim_{z\to 0-}\frac{a(z)}{b(z)}=1$$
and
\begin{align*}
\lim_{z\to 0-}\xi=0,\quad \lim_{z\to
0-}\frac{a(z)}{b(z)}(1+za(z))(\frac{n-\mu}{n})=1-\frac{\mu}{n}
\end{align*}
 for $n=1,2$.
Let $n=1$. Then
\begin{align*}
\delta_r(\lambda,\mu;z)= (a(z)-1-za(z))
\overline{\mathcal{H}}_{0}(\lambda,\mu)+
\frac{a(z)}{b(z)}(1+za(z))({1-\mu})+\xi
\end{align*}
and the corollary  follows for $n=1$. Let $n=2$. In a similar manner
to the case of $n=1$ we have
\begin{align*}
\delta_r(\lambda,\mu;z)= (a(z)-\frac{1+za(z)}{2})
\overline{\mathcal{H}}_{0}(\lambda,\mu)+
\frac{a(z)}{b(z)}(1+za(z))(1-\frac{\mu}{2})+\xi,
\end{align*}
and the corollary  follows for $n=2$. Hence the proof of the
corollary can be derived.
\end{proof}

We define $\bar\delta_r(\lambda,\mu;z)$ for $z\in(-\infty,0]$ by
\begin{align}
\bar\delta_r(\lambda,\mu;z)=\left\{\begin{array}{ll}
\delta_r(\lambda,\mu;z),&z\in(-\infty, 0),\\
\lim_{z\to 0-} \delta_r(\lambda,\mu;z), &z=0.
\end{array}
\right.
\end{align}
From Corollary \ref{delta1 con} we can see that $\bar
\delta_r(\lambda,\mu;z)$ converges to
\begin{align}
\label{del} \bar \delta_r(\lambda,\mu;0)=\left\{
\begin{array}{lll}1-\mu,&n=1, &(\lambda,\mu)\in \overline{\mathfrak H}_0,\\
1-\mu/2,&n=2,&(\lambda,\mu)\in \overline{\mathfrak H}_0,\\
0,&n\geq3,& (\lambda,\mu)\in \overline{\mathfrak H}_0.
\end{array}
\right.
\end{align}
\begin{remark}
We give a remark on \eqref{del}. Let $n=1,2$. If $(\lambda,\mu)\in
\overline{\mathfrak H}_0$, then $(1-\lambda)(1-\mu/n)=1$ is
satisfied by \eqref{f2}, which implies that $1-\mu\not=0$ for $n=1$,
$1-\mu/2\not=0$ for $n=2$.
\end{remark}
We can also show the continuity of $\bar \delta_r(\lambda,\mu;z)$ on
$z$, which is summarised in the lemma below.
\begin{lemma}
It follows that
\begin{description}
\item[($n=1,2$)]
$\bar \delta_r(\lambda,\mu;z)$ is continuous in $z\in(-\infty,0]$
for $(\lambda,\mu)\in \overline{\mathfrak H}_{0}$,
\item[($n\geq3$)]
$\bar \delta_r(\lambda,\mu;z)$ is continuous in $z\in(-\infty,0]$
for $(\lambda,\mu)\in {\mathbb R}^2$.
\end{description}
\end{lemma}
Let $z=0$. Then the asymptote of the hyperbola ${\mathfrak H}_{0}$
is given by
$$
(\lambda_{\infty}(0),\mu_{\infty}(0))=\left\{\begin{array}{ll}
(1,n)&n=1,2,\\
(\frac{a(0)}{b(0)},n)&n\geq3.
\end{array}
\right.$$
 The brunches  $ \Gamma_l(0)$ and $\Gamma_r(0)$
of the hyperbola $\overline{\mathfrak H}_{0}$ split $\mathbb{R}^2$
into three open sets
\begin{align*}
&G_{0}=\{(\lambda,\mu)\in \mathbb{R}^2;
{\overline{\mathcal{H}}}_{0}(\lambda,\mu)>0,
\lambda<\lambda_{\infty}(0)
 \},\quad
\\
&G_{1}=\{(\lambda,\mu)\in \mathbb{R}^2;
{\overline{\mathcal{H}}}_{0}(\lambda,\mu)<0
 \},\quad
\\
&G_{2}=\{(\lambda,\mu)\in \mathbb{R}^2;
{\overline{\mathcal{H}}}_{0}(\lambda,\mu)>0,
\lambda>\lambda_{\infty}(0)
 \}.
\end{align*}
We set
$$\Gamma_l=\Gamma_l(0),\quad \Gamma_r=\Gamma_r(0)$$
for notational simplicity. Hence $\partial G_0=\Gamma_l$ and
$\partial G_2=\Gamma_r$ follow from the definition of $G_0$ and
$G_2$. See Figure~\ref{f4}. We can check the value of $\bar
\delta_r(\lambda,\mu;z)$ for each $(\lambda,\mu)$ and $z\in
(-\infty,0]$ in the next lemma.
\begin{figure}[h]
\begin{center}
\begin{tikzpicture}[scale=1] 
\useasboundingbox(-1,-2) rectangle(4,4);

\draw[thick](-1,0)--(4,0)node[below]{$\lambda$};
\draw[thick](0,-1)--(0,4)node[left]{$\mu$};

\draw[densely dashed](1,-1)node[below]{$\frac{a(0)}{b(0)}$}--(1,4);
\draw[densely dashed](-1,1)node[left]{$n$}--(4,1);

\draw[very thick](-1,0.8)to[out=50,in=140,relative](0.8,-1);
\draw[very thick](1.2,4)to[out=320,in=220,relative](4,1.2);

\node[scale=1] at(-0.5,-0.5){$G_0$}; \node[scale=1]
at(1,1.5){$G_1$}; \node[scale=1] at(3,3){$G_2$}; \node
at(-1,0.5){$\Gamma_l$}; \node at(1.7,3.8){$\Gamma_r$};

\end{tikzpicture}
\end{center}
\vspace{-0.5cm}\caption{Region of $G_j$ for $n\geq 3$} \label{f4}
\end{figure}
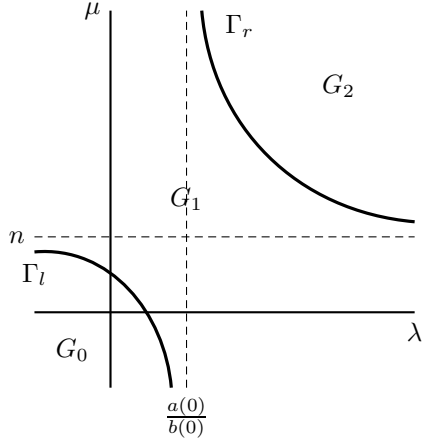

{
\begin{lemma}\label{F1}
We have the following facts:
\begin{itemize}
\item[(a)]
\begin{enumerate}
\item
Let $(\lambda,\mu)\in G_{0}\cup \Gamma_l$. Then $\bar
\delta_r(\lambda,\mu;z)\not=0$ for  $z\in (-\infty,0)$.
\item
Let $(\lambda,\mu)\in \Gamma_l$. Then
    $\bar \delta_r(\lambda,\mu;0)\neq 0$  for $n=1,2$.
\item
Let $(\lambda,\mu)\in \Gamma_l$. Then
  $\bar \delta_r(\lambda,\mu;0)=0$ for $n\geq 3$.
\item
Let $(\lambda,\mu)\in G_{0}$. Then $\bar
\delta_r(\lambda,\mu;0)\not=0$ for $n\geq1$.
\end{enumerate}
\item[(b)]
\begin{enumerate}
\item
  Let $(\lambda,\mu)\in G_{1}\cup \Gamma_r$. Then
there exists unique point $z\in (-\infty,0)$ such that $\bar
\delta_r(\lambda,\mu;z)=0$.
\item
Let $(\lambda,\mu)\in \Gamma_r$. Then
    $\bar \delta_r(\lambda,\mu;0)\neq 0$  for $n=1,2$.
\item
Let $(\lambda,\mu)\in \Gamma_r$. Then
 $\bar \delta_r(\lambda,\mu;0)=0$ for $n\geq 3$.
\end{enumerate}
\item[(c)]  Let $(\lambda,\mu)\in G_{2}$. Then
there exist two zeros $z_1, z_2\in (-\infty,0)$ such that $\bar
\delta_r(\lambda,\mu;z_1)=\bar \delta_r(\lambda,\mu;z_2)=0$.
\end{itemize}
\end{lemma}
}\begin{proof} Let $(\lambda,\mu)\in G_{0}\cup \Gamma_l$. Then we
can see that $(\lambda,\mu)\not\in \overline{\mathfrak{H}}_z$ for
any $z\in(-\infty,0)$. Thus (a)(1) follows.

Also by  \eqref{sch}, it can be seen that $\cup_{z\in(-\infty,
0)}\Gamma_l(z)\supset G_1$, $\Gamma_l(z)\cap \Gamma_l(w)=\emptyset$
if $z\not=w$, and $\Gamma_r(z) \cap G_1=\emptyset $.
Hence
there exists a unique
 $z\in (-\infty,0)$  such that
$ (\lambda,\mu)\in \Gamma_l(z) $, which proves (b)(1).
We can also see that
 $\cup_{z\in(-\infty, 0)}\Gamma_l(z)\supset G_2$,
 $\cup_{z\in(-\infty, 0)}\Gamma_r(z)\supset G_2$,
$\Gamma_l(z)\cap \Gamma_l(w)=\emptyset$ if $z\not=w$,
$\Gamma_r(z)\cap \Gamma_r(w)=\emptyset$ if $z\not=w$, and
$\Gamma_l(z)\cap \Gamma_r(z)=\emptyset$.
Hence there exist $z_1,z_2\in (-\infty,0)$  such
 that
 $
(\lambda,\mu)\in \Gamma_l(z_1)$ and $ (\lambda,\mu)\in \Gamma_r(z_2)
 $,
which proves (c). 
We note that since $(\lambda,\mu)\in \Gamma_l$ implies that
$1\not=\mu$ for $n=1$, and $2\not=\mu$ for $n=2$, $\bar
\delta_r(\lambda,\mu;0)\neq 0$  for $n=1,2$ follows. Hence (a)(2)
and (a)(3) follow from \eqref{del}, and (b)(2) and (b)(3) are
similarly proven. Finally for $(\lambda,\mu)\in G_0$ we have
$$\bar\delta_r(\lambda,\mu;0)=\left\{\begin{array}{ll}\infty,&n=1,2,\\
b(0)\bar{\mathcal{H}}_0(\lambda,\mu)\not=0, &n\geq 3.
\end{array}\right.$$
Then (a)(4) follows.
\end{proof}
\subsubsection{Zeros of $\delta_c(\lambda;z)$}
We study zeros of $\delta_c(\lambda;z)$.
 In a similar manner we extend $\delta_c(\lambda,z)$ for $z\in (-\infty,0]$.
When $n\geq2$,   the  function $c(z)-d(z)$ exists, and due to its
  monotone property (See Lemma \ref{LemMon}) we can define
$
 \alpha = 
 \lim_{z\to 0-}c(z)-d(z).
$ Note that $\alpha>0$ and we set
$$
\lambda_c=\frac{1}{\alpha}.
$$
 Let us write $
\delta_c(\lambda;z)=\varrho(\lambda;z)^{n-1} $, where
$\varrho(\lambda;z)=\lambda (c(z)-d(z))-1$. We define $\bar
\delta_c(\lambda;z)$ by
\begin{align}\label{K}
\bar \delta_c(\lambda;z)=\left\{
\begin{array}{lll}
\delta_c(\lambda;z),&z\in (-\infty, 0),&n\geq 1,\\
(\lambda\alpha-1)^{n-1},&z=0, &n\geq2,\\
1,&z=0, &n=1.
\end{array}\right.
\end{align}

\begin{lemma}\label{delta cos}
Let $n\geq 2$.  Then (a)--(c) follow.
\begin{itemize}
\item[(a)] Let $\lambda\leq \lambda_c$. Then
$\varrho(\lambda;z)\not=0$ for any $z\in (-\infty,0)$.
\item[(b)] Let $\lambda= \lambda_c$. Then
$\varrho(\lambda;0)=0$.
\item[(c)] Let $\lambda> \lambda_c$.
Then there exists unique $z\in (-\infty,0)$ such that
$\varrho(\lambda;z)=0$ with multiplicity one.
\end{itemize}
\end{lemma}
\begin{proof}
Since  $c(z)-d(z)>0$ is strictly monotonously increasing in
$(-\infty,0)$, we get
\begin{align*}
& \varrho(\lambda;z)\leq
\varrho(\lambda_c;z)<\varrho(\lambda_c;0)=0,\quad
\mbox{if}\quad 0<\lambda\leq \lambda_c,\\
& \varrho(\lambda;z)=-1,\quad \mbox{if}\quad \lambda=0,
\end{align*}
which prove (a) and (b). Since  $\varrho(\lambda;0)>
\varrho(\lambda_c;0)=0$ and $\lim_{z\to -\infty}
\varrho(\lambda;z)=-1$ there exists $ z\in (-\infty,0)$ such that
$\varrho(\lambda;z)=0$. By the monotonicity of
$\varrho(\lambda;\cdot)$ this zero is a unique and has multiplicity
one. Hence (c) is proven.
\end{proof}

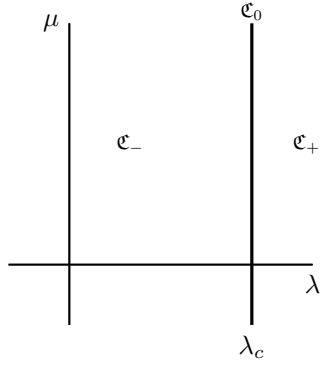
\begin{figure}[h]
\begin{center}
\begin{tikzpicture}[scale=0.8]
\useasboundingbox(-1,-2) rectangle(4,4);
\draw[thick](-1,0)--(4,0)node[below]{$\lambda$};
\draw[thick](0,-1)--(0,4)node[left]{$\mu$};
\node[below]at(1.5,-1.5){$ $}; \draw[thin](-1,0)--(4,0);
\draw[thin](0,-1)--(0,4); \draw[very
thick](3,-1)node[below]{$\lambda_c$}--(3,4);
\node[scale=0.8]at(1,2){${{\mathfrak C}_-}$};
\node[scale=0.8]at(3.9,2){${{\mathfrak C}_+}$};
\node[scale=0.8]at(3,4.2){${{\mathfrak C}_0 }$};
\end{tikzpicture}
\end{center}
\vspace{-0.5cm} \caption{Regions of $C_\pm$  for $n\geq 2$}
\label{f81}
\end{figure}
We divide $(\lambda,\mu)$-plane into two half planes $C_\pm$ and the
boundary ${\mathfrak C}_0 $. Set
$${\mathfrak C}_-=\{(\lambda,\mu)\in{\mathbb R}^2; \lambda<\lambda_c\},
{\mathfrak C}_0 =\{(\lambda,\mu)\in {\mathbb R}^2;
\lambda=\lambda_c\}, {\mathfrak C}_+=\{(\lambda,\mu)\in{\mathbb
R}^2; \lambda>\lambda_c\}.$$ See Figure \ref{f81}. We immediately
have a lemma.
\begin{lemma}\label{FFF}
Let $n\geq2$. Then (a)--(c) follow.
\begin{itemize}
\item[(a)] For any $(\lambda,\mu)\in {\mathfrak C}_-\cup {\mathfrak C}_0 $,  
$\bar \delta_c(\lambda;\cdot)$ has no zero 
in $(-\infty,0)$.
\item[(b)] Let  $(\lambda,\mu)\in {\mathfrak C}_0 $. Then
 $\bar \delta_c(\lambda;0)=0$, and
$z=0$ has multiplicity $n-1$.
\item[(c)] For any $(\lambda,\mu)\in {\mathfrak C}_+$, 
$\bar \delta_c(\lambda;\cdot)$ has a unique zero in $(-\infty,0)$
with multiplicity $n-1$.
\end{itemize}
\end{lemma}
\begin{proof}
This follows from Lemma \ref{delta cos}.
  \end{proof}

\subsection{Eigenvalues of $H_{\lambda\mu}^{\rm e}$}
In the previous sections we consider zeros of $\bar
\delta_r(\lambda,\mu;z)$ and $\bar \delta_c(\lambda;z)$ for
$z\in(-\infty,0)$. Let $(\lambda,\mu)$ and $z\in (-\infty,0)$ be
solution of $\bar \delta_r(\lambda,\mu;z)=0$ or $\bar
\delta_c(\lambda;z)=0$. Then $z\in \sigma_p(H_{\lambda\mu}^{\rm
e})$. In this section we summarise spectral properties of
$H_{\lambda\mu}^{\rm e}$ derived from zeros of $\bar
\delta_r(\lambda,\mu;z)\bar \delta_c(\lambda;z)$.
\begin{definition}[Threshold eigenvalue  and threshold resonance]
\label{reson def even} Let $f$ be a solution of
$H_{\lambda\mu}^{\mathrm{e}} f=0$ (resp. $H_{\lambda}^{\mathrm{o}}
f=0$).
\begin{itemize}
\item[(1)] If $f\in L^2_{\rm e}(\mathbb{T} ^n)$ (resp. $f\in L^2_{\rm o}(\mathbb{T} ^n)$), we say that $0$ is
a lower threshold eigenvalue of $H_{\lambda\mu}^{\mathrm{e}}$ (resp.
$H_{\lambda}^{\mathrm{o}}$).
\item[(2)]
If $f\in L^1_{\rm e}(\mathbb{T} ^n)\setminus L^2_{\rm e}(\mathbb{T}
^n)$ (resp. $f\in L^1_{\rm o}(\mathbb{T} ^n)\setminus L^2_{\rm
o}(\mathbb{T} ^n)$), we say that $0$ is a lower  threshold resonance
of $H_{\lambda\mu}^{\mathrm{e}}$ (resp. $H_{\lambda}^{\mathrm{o}}$).
\item[(3)]
If $f\in L^\epsilon_{\rm e}(\mathbb{T} ^n)\setminus L^1_{\rm
e}(\mathbb{T} ^n)$ (resp. $f\in L^\epsilon_{\rm o}(\mathbb{T}
^n)\setminus L^1_{\rm o}(\mathbb{T} ^n)$) for any $0<\epsilon<1$, we
say that $0$ is a lower  super-threshold resonance of
$H_{\lambda\mu}^{\mathrm{e}}$ (resp. $H_{\lambda}^{\mathrm{o}}$).
\end{itemize}
\end{definition}
In what follows we may use "threshold eigenvalue" (resp. threshold
resonance) instead of "lower threshold eigenvalue" (resp. lower
threshold resonance)
 for simplicity.
Then $(\lambda,\mu)$-plane is divided into 11 regions. $G_0$,
$\Gamma_l$, $G_1\cap {\mathfrak C}_-$, $G_1\cap {\mathfrak C}_0 $,
$G_1\cap {\mathfrak C}_+$, $\Gamma_r\cap {\mathfrak C}_-$,
$\Gamma_r\cap {\mathfrak C}_0 $, $\Gamma_r\cap {\mathfrak C}_+$,
$G_2\cap {\mathfrak C}_-$, $G_2\cap {\mathfrak C}_0 $ and $G_2\cap
{\mathfrak C}_+$.
\begin{lemma}[Eigenvalues of $H_{\lambda\mu}^{\rm e}$  for $n\geq 2$]
\label{lemma Delta 0} Let $n\geq2$. Then $H_{\lambda\mu}^{\rm e}$
has the following facts:
\begin{itemize}
\item[(1)] $(\lambda,\mu)\in G_0$.
There is no eigenvalue in $(-\infty,0)$, and there is neither
threshold eigenvalue  nor threshold resonance.
\item[(2)] $(\lambda,\mu)\in \Gamma_l$.
\begin{itemize}
\item[$n=2$]
There is no eigenvalue in $(-\infty,0)$ and there is neither
threshold eigenvalue  nor threshold resonance.
\item[$n\geq 3$]
There is no eigenvalue in $(-\infty,0)$ but there is a simple
threshold eigenvalue  or threshold resonance.
\end{itemize}
\item[(3)] $(\lambda,\mu)\in G_1\cap {\mathfrak C}_-$.
There is a simple eigenvalue in $(-\infty,0)$ but there is neither
threshold eigenvalue  nor threshold resonance.
\item[(4)] $(\lambda,\mu)\in G_1\cap {\mathfrak C}_0 $.
There is a simple eigenvalue in $(-\infty,0)$ and there is an
$(n-1)$-fold threshold eigenvalue  or threshold resonance.
\item[(5)] $(\lambda,\mu)\in G_1\cap {\mathfrak C}_+$.
There are a simple eigenvalue and an $(n-1)$-fold eigenvalue in
$(-\infty,0)$, but there is neither threshold eigenvalue  nor
threshold resonance.
\item[(6)] $(\lambda,\mu)\in \Gamma_r\cap {\mathfrak C}_-$.
\begin{itemize}
\item[$n=2$]
There is a simple eigenvalue in $(-\infty,0)$ but there is neither
threshold eigenvalue  nor threshold resonance.
\item[$n\geq3$]
There is a simple eigenvalue in $(-\infty,0)$ and there is a simple
threshold eigenvalue  or threshold resonance.
\end{itemize}
\item[(7)] $(\lambda,\mu)\in \Gamma_r\cap {\mathfrak C}_0 $.
\begin{itemize}
\item[$n=2$]
There is a simple eigenvalue in $(-\infty,0)$ and  there is an
$(n-1)$-fold threshold eigenvalue  or  threshold resonance.
\item[$n\geq3$]
There is a simple eigenvalue in $(-\infty,0)$, and there are an
$(n-1)$-fold threshold eigenvalue  or  threshold resonance, and a
simple threshold eigenvalue  or  threshold resonance.
\end{itemize}
\item[(8)] $(\lambda,\mu)\in \Gamma_r\cap {\mathfrak C}_+$.
\begin{itemize}
\item[$n=2$]
There is a simple eigenvalue and an $(n-1)$-fold eigenvalue in
$(-\infty,0)$, but there is neither threshold eigenvalue   nor
threshold resonance.
\item[$n\geq3$]
There are a simple eigenvalue and an $(n-1)$-fold eigenvalue in
$(-\infty,0)$. There is a  simple threshold eigenvalue  or
threshold resonance.
\end{itemize}
\item[(9)] $(\lambda,\mu)\in G_2\cap {\mathfrak C}_-$.
There are two eigenvalues in $(-\infty,0)$ but there is neither
threshold eigenvalue  nor threshold resonance.
\item[(10)] $(\lambda,\mu)\in G_2\cap {\mathfrak C}_0 $.
There are two eigenvalues in $(-\infty,0)$ and there is an
$(n-1)$-fold threshold eigenvalue  or  threshold resonance.
\item[(11)]
$(\lambda,\mu)\in G_2\cap {\mathfrak C}_+$. There are three
eigenvalues in $(-\infty,0)$ and one of them is $(n-1)$-fold, but
there is neither threshold eigenvalue  nor threshold resonance.
\end{itemize}
\end{lemma}
\begin{proof}
This lemma follows from Lemmas \ref{F1} and  \ref{FFF}, and the fact
that $z\not=0$ is an eigenvalue if and only if $\bar
\delta_r(\lambda,\mu;z)\bar \delta_c(\lambda;z)=0$, and $0$ is an
threshold eigenvalue  or threshold resonance   if and only if $\bar
\delta_r(\lambda,\mu;0)\bar \delta_c(\lambda;0)=0$.
\end{proof}
By virtue of Lemma \ref{F1}, $\bar \delta_r(\lambda,\mu;\cdot)$ has
at most two zeros in $(-\infty,0)$ for $(\lambda,\mu)\in G_2$ or
$(\lambda,\mu)\in G_1\cup\Gamma_r$. Now we can see the explicit form
of  these eigenvectors.
In the case of $n=1$ we know that $\delta_c(\lambda;z)=1$. Hence
zeros of $\delta_r(\lambda,\mu;z)\delta_c(\lambda;z)$ coincides with
those of $\delta_r(\lambda,\mu;z)$. We  have the lemma.

\begin{lemma}[Eigenvalues of $H_{\lambda\mu}^{\rm e}$ for $n=1$]
\label{toto} We have the following facts:
\begin{itemize}
\item[(1)] $(\lambda,\mu)\in G_0\cup \Gamma_l$.
There is no eigenvalues in $(-\infty,0)$, and there is neither
threshold eigenvalue  nor threshold resonance.
\item[(2)] $(\lambda,\mu)\in G_1\cup \Gamma_r$.
There is a simple eigenvalue in $(-\infty,0)$, but  there is neither
threshold eigenvalue  nor threshold resonance.
\item[(3)] $(\lambda,\mu)\in G_2$.
There are two  eigenvalues in $(-\infty,0)$, but there is neither
threshold eigenvalue  nor threshold resonance.
\end{itemize}
\end{lemma}
\begin{proof}
This lemma follows from Lemmas \ref{F1} and the fact that $z\not=0$
is an eigenvalue if and only if $\bar \delta_r(\lambda,\mu;z)=0$,
and $0$ is an threshold eigenvalue  or threshold resonance   if and
only if $\bar \delta_r(\lambda,\mu;0)$.
\end{proof}

\begin{lemma}\label{sora1}
Let $n\geq1$. (1) Let $\lambda\not=0$. We assume  that $z_1,z_2\in
(-\infty,0)$ and $\delta_r(\lambda,\mu;z_k)=0$ (if they exist). Then
$1-\mu a(z_k)\not=0$ for $k=1,2$ and $ G_{\rm e}(z_k)Z_k=Z_k$ has
the solutions:
$$
Z_k= \VVV
{\frac{\lambda}{\sqrt{2}} \frac{nb(z_k)}{1-\mu a(z_k)}\\
1\\
\vdots\\
1},\quad k=1,2,$$ and the corresponding eigenfunctions,
$H_{\lambda\mu}^{\mathrm{e}} f_k=zf_k$,  are
\begin{align}
\label{e1}
    f_k(p)=\frac{ \lambda}{\sqrt{2}}\frac{1}{(2\pi) }
    \frac{1}{E(p)-z_k}
    \Big (\mu \frac{  n  b(z_k)}{1-\mu a(z_k)}
+\sum_{j=1}^n \cos p_j \Big ),\quad k=1,2.
\end{align}
(2) Let $\lambda=0$. We assume  that $z\in (-\infty,0)$ and
$\delta_r(0,\mu;z)=0$. Then $1-\mu a(z)=0$  and $ G_{\rm e}(z)Z=Z$
has the solution:
$$
Z=\VVV{1\\ \sqrt2\mu b(z)\\ \vdots\\ \sqrt2\mu b(z)}
$$
and the corresponding eigenfunction, $H_{\lambda\mu}^{\mathrm{e}}
f=zf$,  is
\begin{align}
\label{e11}
    f(p)=\frac{\mu}{(2\pi) }\frac{1}{E(p)-z}
\end{align}
\end{lemma}
\begin{proof}
We prove the case of $n\geq 2$. The proof for the case of $n=1$ is
similar. Since $\delta_r(\lambda,\mu;z)=0$, we see that
$$
\big(1-\mu a(z)\big)\Big(1-\lambda\big(c(z)+(n-1)d(z)\big)\Big)-n
\lambda\mu b^2(z)=0.$$ Then $1-\mu a(z)\not=0$ if and only if
$\lambda\not=0$, and we also have the algebraic relation
$$1-\lambda\big(c(z)+(n-1)d(z)\big)=\frac{n\lambda \mu b^2(z)}{1-\mu a(z)}.$$
From this relation it follows that
$$G_{\rm e}(z_k)
\VVV
{\frac{\lambda}{\sqrt{2}} \frac{nb(z_k)}{1-\mu a(z_k)}\\
1\\
\vdots\\z_k 1} = \VVV
{\frac{\lambda}{\sqrt{2}} \frac{nb(z_k)}{1-\mu a(z_k)}+\frac{\lambda}{\sqrt{2}}nb(z_k)\\
\lambda\frac{n\mu b(z_k)^2}{1-\mu a(z_k)}+\lambda c(z_k)+\lambda(n-1) d(z_k)\\
\vdots\\
\lambda\frac{n\mu b(z_k)^2}{1-\mu a(z_k)}+\lambda
c(z_k)+\lambda(n-1) d(z_k)}= \VVV
{\frac{\lambda}{\sqrt{2}} \frac{nb(z_k)}{1-\mu a(z_k)}\\
1\\
\vdots\\
1}.
$$
Then $G_{\rm e}(z_k)Z_k=Z_k$ for $\lambda\not=0$. In the case of
$\lambda=0$ we can also see that
$$G_{\rm e}(z)
\VVV
{1\\
\sqrt 2 \mu b(z)\\
\vdots\\
\sqrt 2 \mu b(z)} = \VVV
{\mu a(z)\\
\sqrt 2 \mu b(z)\\
\vdots\\
\sqrt 2 \mu b(z)}= Z.
$$
Then the lemma is proven.
\end{proof}

Next we show the eigenfunction corresponding to zeros of
$\delta_c(\lambda;\cdot)$.
\begin{lemma}\label{320}
Let $n\geq 2$,  $z\in (-\infty,0)$ and $\bar \delta_c(\lambda;z)=0$.
I.e., $\lambda=\frac{1}{c(z)-d(z)}$. Then the solutions of $ G_{\rm
e}(z)Z=Z$ are given by
\begin{align}
\label{ev} Z_1=\VVV{0\\ 1\\-1\\0\\ \vdots\\0}, Z_2=\VVV{0\\
1\\0\\-1\\ \vdots\\0}, \cdots Z_{n-1}=\VVV{0\\ 1\\0\\0\\
\vdots\\-1},
\end{align}
and  hence
the corresponding eigenfunctions,
$H_{\lambda\mu}^{\mathrm{e}} g_j=zg_j $,  are 
\begin{align}\label{e2}
    g_j(p)=\frac{\lambda}{\sqrt 2} \frac{1}{(2\pi) }
    \frac{1}{E(p)-z}(\cos p_1-\cos p_{j+1}),\quad j=1,\dots,n-1.
\end{align}
In particular the multiplicity of eigenvalue $z$ is at least $n-1$.
\end{lemma}
\begin{proof}
Since  $\lambda(c(z)-d(z))=1$, we see that
$$G_{\rm e}(z)
\VVV
{0\\
1\\
-1\\
0\\
\vdots\\
0 } = \VVV
{0\\
\lambda(c(z)-d(z))\\
\lambda(d(z)-c(z))\\
0\\
\vdots\\
0 } = \VVV
{0\\
1\\
-1\\
0\\
\vdots\\
0 }.
$$
Then $G_{\rm e}(z) Z_1=Z_1$. In the same way as $Z_1$ we can see
that $G_{\rm e}(z) Z_j=Z_j$ for $j=2,...,n-1$. Then the lemma is
proven.
\end{proof}

\subsection{Threshold eigenvalues and threshold resonances  for $H_{\lambda\mu}^{\mathrm{e}}$}
\label{subsec reso} In this section we study the spectrum located on
the left edge of the essential spectrum $[0,2n]$, i.e., $z=0$.
Suppose that $(\lambda,\mu)\in \mathfrak{H}_0$. Then it is possibly
$\bar \delta_r(\lambda,\mu;0)=0$ or $\bar
\delta_c(\lambda,\mu;0)=0$. By Corollary \ref{delta1 con} and
\eqref{K} we see that for $(\lambda,\mu)\in \mathfrak{H}_0$
\begin{align}
\begin{array}{ll}
\bar \delta_r(\lambda,\mu;0)\not=0\not=
1=\bar \delta_c(\lambda,\mu;0),&n=1,\\
\bar \delta_r(\lambda,\mu;0)\not=0,&n=2.
\end{array}
\end{align}
Hence we study zeros of $\bar\delta_r(\lambda,\mu;0)$ for $n\geq 3$,
and those of $\bar\delta_c(\lambda;0)$ for $n\geq 2$. We  set
$a(0)=a$ and $b(0)=b$, and both $a$ and $b$ are finite for $n\geq
3$. In these case however the proofs are similar to those of Lemmas
\ref{sora1} and \ref{320} where we discuss eigenvalues in
$(-\infty,0)$.
\begin{lemma}
\label{sora2} Let $n\geq 3$. (1) Let $\lambda\not=0$ and $\bar
\delta_r(\lambda,\mu;0)=0$. Then $1-\mu a\not=0$ and $ G_{\rm
e}(0)Z=Z$ has the  solution
$$
Z=\VVV
{\frac{\lambda}{\sqrt{2}} \frac{nb}{1-\mu a}\\
1\\
\vdots\\
1}
$$
and the corresponding equation $H_{\lambda\mu}^{\mathrm{e}} f=0$ has
the solution:
\begin{align}
\label{z0}
    f(p)=\frac{ \lambda}{\sqrt{2}}
    \frac{1}{(2\pi) }
    \frac{1}{E(p)}
    \Big (\mu \frac{  n  b}{1-\mu a}
+\sum_{j=1}^n \cos p_j \Big ).
\end{align}
(2) Let $\lambda=0$ and
 $\delta_r(0,\mu;0)=0$.
Then $1-\mu a=0$  and $ G_{\rm e}(0)Z=Z$ has the  solution:
$$
Z=\VVV{1\\ \sqrt 2 \mu b\\ \vdots\\ \sqrt 2 \mu b}
$$
and the corresponding equation $H_{\lambda\mu}^{\mathrm{e}} f=0$ has
the solution:
\begin{align}
\label{z1}
    f(p)=\frac{\mu}{(2\pi) }\frac{1}{E(p)} .
\end{align}
\end{lemma}
\begin{proof}
Replacing $z$ in Lemma \ref{sora1} with $0$ we can prove the lemma
in the same way as that of  Lemma \ref{sora1}.
\end{proof}

Next we show the solution corresponding to zeros of
$\delta_c(\lambda;\cdot)$.
Similar to the case of $\delta_r(\lambda,\mu;z)=0$, we have the
lemma below.
\begin{lemma} \label{n res lemm}
Let $n\geq 2$ and  $\delta_c(\lambda;0)=0$, i.e.,
$\lambda=\lambda_c$. Then the solutions of $ G_{\rm e}(0)Z=Z$
are given by \eqref{ev} and  hence
the corresponding equation $H_{\lambda\mu}^{\mathrm{e}} g_j=0$ has
the solutions
\begin{align}\label{z2}
    g_j(p)=\frac{\lambda_c}{\sqrt 2} \frac{1}{(2\pi) } \frac{1}{E(p)}(\cos p_1-\cos p_{j+1}),\quad j=1,\dots,n-1.
\end{align}
\end{lemma}
\begin{proof}
Replacing $z$ in Lemma \ref{320} with $0$ we can prove the lemma in
the same way as that of  Lemma \ref{320}.
\end{proof}

Recall that
\begin{equation}
\label{integ finite} u_0=\frac{1}{(2\pi)^n}
\int_{\mathbb{T}^n}f(p)dp,\quad
u_j=\frac{1}{(2\pi)^n}\int_{\mathbb{T}^n} \cos p_j f(p)dp,\quad j=1,
...,n.
 \end{equation}
As was seen above the problem for $n\geq 3$ can be reduced to study
the spectrum  of $G_{\rm e}$ by the Birman-Schwinger principle, the
problem for $n=1,2$ should be however directly investigated.
\begin{lemma}\label{n=1 res lemm} Let $n=1$.
\begin{itemize}
\item[(1)] Suppose that $f\in L^1(\mathbb T)$ and
$H_{\lambda\mu}^{\mathrm{e}} f =0$. Then $f=0$. In particular
$H_{\lambda\mu}^{\mathrm{e}}$ has no threshold resonance.
\item[(2)] There is no non-zero $f$ such that
$f\in L^\epsilon(\mathbb T^2)\setminus L^1(\mathbb T^2)$ for some
$0< \epsilon<1$ and $H_{\lambda\mu}^{\mathrm{e}} f =0$. In
particular $H_{\lambda\mu}^{\mathrm{e}}$ has no super-threshold
resonance.
\end{itemize}
\end{lemma}
\begin{proof}
(1) $H_{\lambda\mu}^{\mathrm{e}} f=0$ gives $f=\varphi/E$ and
$\varphi(p)=\mu u_0 +\lambda u_1 \cos p$ by \eqref{Hf=0}. From $f\in
L^1(\mathbb{T})$ it follows that $\varphi(0)=\mu u_0 +\lambda u_1
=0$. Hence
$$
f(p)=\frac{1}{E(p)}( 1 -\cos p)\mu u_0=\mu u_0.$$ Substituting this
into the second term in   \eqref{integ finite}, we get $ u_1=\mu u_0
\frac{1}{2\pi}\int_{\mathbb{T}}\cos t dt=0$,
which gives  $\mu u_0=0$ and $f=0$. \\
(2) Since $f\not\in L^1(\mathbb T)$. It must be that $\mu=0$ and $
f={\varphi}/{E}$ with $\varphi(p)=\lambda u_1 \cos p$. Hence
\begin{align*}
u_1=\frac{\lambda }{(2\pi)^2} \int_{\mathbb{T}} \frac{u_1\cos^2
p}{E(p)}dp.
\end{align*}
Then $u_1=0$, since $\int_{\mathbb{T}} \frac{\cos^2 p}{E(p)}dp
=\infty$. Then $f=0$ follows.
\end{proof}

Next we discuss the spectrum of $H_{\lambda\mu}^{\mathrm{e}}$ for
$n=2$ at the lower threshold
\begin{lemma}
\label{n=2 res lemm} Let $n=2$.
\begin{itemize}
\item[(1)] Suppose that $f\in L^1(\mathbb T^2)$ and
$H_{\lambda\mu}^{\mathrm{e}} f =0$. Then $(\lambda,\mu)\in
{\mathfrak C}_0 $ and
\begin{align}\label{F}
f(p)=\lambda_c  u_1\frac{\cos p_1-\cos p_2}{E(p)}.
\end{align}
In particular $f \in L^2(\mathbb T^2)$ and $H_{\lambda\mu}^{\mathrm{e}}$ has no threshold resonance. \\
\item[(2)] There is no non-zero $f$ such that
$f\in L^\epsilon(\mathbb T^2)\setminus L^1(\mathbb T^2)$ for some
$0< \epsilon<1$ and $H_{\lambda\mu}^{\mathrm{e}} f =0$. In
particular $H_{\lambda\mu}^{\mathrm{e}}$ has no super-threshold
resonance.
\end{itemize}
\end{lemma}
\begin{proof}
(1) Consider $H_{\lambda\mu}^{\mathrm{e}} f=0$ in
$L^1(\mathbb{T}^2)$. We can take $ f={\varphi}/{E}$ and
$\varphi(p)=\mu u_0 +\lambda u_1 \cos p_1 +\lambda u_2 \cos p_2$.
Since $f\in L^1(\mathbb{T}^2)$, we get $\varphi(0)=\mu u_0
+\lambda(u_1+u_2)  =0 $ and so
$$
f(p)=\frac{\lambda}{E(p)}\left(   -u_1( 1 -\cos p_1) -u_2( 1 -\cos
p_2)\right)
$$
By
\eqref{integ finite} we obtain
\begin{align*}
u_1&=-\frac{\lambda }{(2\pi)^2}\left (
u_1\int_{\mathbb{T}^2} \frac{\cos p_1( 1 -\cos p_1)}{E(p)}dp +u_2\int_{\mathbb{T}^2} \frac{\cos p_1( 1 -\cos p_2)}{E(p)}dp\right),\\
u_2&=-\frac{\lambda }{(2\pi)^2}\left ( u_1\int_{\mathbb{T}^2}
\frac{\cos p_2( 1 -\cos p_1)}{E(p)}dp +u_2\int_{\mathbb{T}^2}
\frac{\cos p_2( 1 -\cos p_2)}{E(p)}dp\right).
\end{align*}
Since $ \int_{\mathbb{T}^2} \frac{\cos p_1( 1 -\cos
p_1)}{E(p)}dp=-\int_{\mathbb{T}^2} \frac{\cos p_1( 1 -\cos
p_2)}{E(p)}dp$, we get
\begin{align*}
u_1&=\frac{\lambda }{(2\pi)^2}\left(u_2-u_1\right)\left (
\int_{\mathbb{T}^2} \frac{\cos p_1( 1 -\cos p_1)}{E(p)}dp\right)
,\\
u_2&=\frac{\lambda }{(2\pi)^2}\left(u_1-u_2\right)\left (
\int_{\mathbb{T}^2} \frac{\cos p_2( 1 -\cos p_2)}{E(p)}dp\right)
\end{align*}
and hence $u_1=-u_2$. Consequently, $\mu u_0=0$, and the solution of
$H_{\lambda\mu}^{\mathrm{e}} f=0$ is of the form
\begin{gather}
\label{toto2} f(p)=\lambda u_1\frac{\cos p_1-\cos p_2}{E(p)}\in
L^2(\mathbb T^2).
\end{gather}
Inserting this into the definition of $u_1$, we have
$\frac{\lambda}{(2\pi)^2}\int_{\mathbb{T}^2}\frac{\cos p_1(\cos
p_1-\cos p_2)}{E (p)} dp=1$ and thus taking $\lambda=\lambda_c$ we
can see that \eqref{F}
is the solution of $H_{\lambda\mu}^{\mathrm{e}} f=0$. Notice that $u_0=0$ follows from \eqref{toto2}.\\
(2) Since $f\not\in L^1(\mathbb T^2)$. It must be that $\mu=0$ and $
f={\varphi}/{E}$ with $\varphi(p)=\lambda u_1 \cos p_1 +\lambda u_2
\cos p_2$. Hence
\begin{align*}
u_1&=\frac{\lambda }{(2\pi)^2}
\int_{\mathbb{T}^2} \frac{u_1\cos^2 p_1+u_2\cos p_1\cos p_2}{E(p)}dp ,\\
u_2&=-\frac{\lambda }{(2\pi)^2} \int_{\mathbb{T}^2} \frac{u_1\cos
p_2\cos p_1+u_2\cos^2 p_2}{E(p)}dp.
\end{align*}
Then $u_1=-u_2$ and $1=\frac{\lambda}{(2\pi)^2} \int_{\mathbb{T}^2}
\frac{\cos p_1(\cos p_1+\cos p_2)}{E(p)}dp$. Thus
$\lambda=\lambda_c$. Then $f$ is given by \eqref{toto2}, but $f\in
L^2(\mathbb T^2)$. This contradicts with $f\not \in  L^1(\mathbb
T^2)$.
\end{proof}

\begin{remark}
The Birman-Schwinger principle is valid for $n\geq 3$, but Lemma
\ref{n res lemm} tells us that the Birman-Schwinger principle is
valid for  $n=2$. Furthermore in Lemma \ref{n=2 res lemm} it can be
seen that
$g_1$ given by \eqref{z2} coincides with \eqref{F}.
\end{remark}


\begin{lemma}[Threshold eigenvalues and threshold resonances of $H_{\lambda\mu}^{\rm e}$] \label{thres and res}
(1)-(5) follow:
\begin{itemize}
\item[(1)] Let $n=1$.
{Then $0$ is none of a   threshold eigenvalue, a   threshold
resonance and a super-threshold resonance.}
\item[(2)] Let $n=2$.
Then $0$ is a   threshold eigenvalue with \eqref{z2}  for
$(\lambda,\mu)\in {\mathfrak C}_0 $ and its multiplicity is one.

\item[(3)] Let $n=3,4$.
Suppose $(\lambda,\mu)\in {\mathfrak{H}_0}$. Then $0$ is a
threshold resonance with eigenvector \eqref{z0}  for
$\lambda\not=0$, and \eqref{z1} for $\lambda=0$, i.e.,
$(\lambda,\mu)=(0,1/a)$.
\item[(4)]
Let $n=3,4$. Suppose $(\lambda,\mu)\in {\mathfrak{H}_0}$. Then $0$
is a   threshold eigenvalue with \eqref{z2}  for $\lambda=\lambda_c$
and its  multiplicity is $n-1$.

\item[(5)]
 Let $n\geq 5$.
Suppose $(\lambda,\mu)\in {\mathfrak{H}_0}$.
 Then
$0$ is  a
  threshold eigenvalue with eigenvector
\eqref{z0}  for $\lambda_c\not=\lambda\not=0$ and  multiplicity one,
\eqref{z0} and \eqref{z2} for $\lambda=\lambda_c$ and  multiplicity
$n$, and \eqref{z1} for $\lambda=0$, i.e., $(\lambda,\mu)=(0,1/a)$,
and  multiplicity one.
\end{itemize}
\end{lemma}
\begin{proof}
(1) follows from Lemma \ref{n=1 res lemm}. The solution of
$H_{\lambda\mu}^{\mathrm{e}} f=0$ is given by \eqref{z0}, \eqref{z1}
and \eqref{z2}. We note that $ \int_{|p|<\epsilon}
\frac{1}{E^2(p)}dp =\infty$ for $n=2, 3,4$ for any $\epsilon>0$,
 and $ \int_{|p|<\epsilon} \frac{1}{E^2(p)}dp <\infty$ for
$n\geq 5$ for any $\epsilon>0$. Since $(\lambda,\mu)\in
{\overline{\mathfrak{H}}}_0$, $n\not=\mu$, and we can see that
$$\frac{nb\mu}{1-\mu a}+\sum_{j=1}^n \cos 0=\frac{nb\mu}{1-\mu a}+n=n\left(\frac{1-\mu(a-b)}{1-\mu a}\right)=
\frac{n-\mu}{1-\mu a}\not=0.$$ Hence, using Lemma \ref{Incl L2}, we
obtain
\begin{align*}
&\eqref{z0},\eqref{z1} \in L^2(\mathbb{T}^n),\quad n\geq 5,\\
&
\eqref{z0},\eqref{z1} \in L^1(\mathbb{T}^n)\setminus L^2(\mathbb{T}^n),\quad n=3,4,\\
&\eqref{z2}\in L^2(\mathbb{T}^n),\quad n\geq2.
\end{align*}
(2) follows from Lemmas \ref{n=2 res lemm} and \ref{n res lemm}. (3)
follows from Lemmas \ref{n res lemm} and  \ref{sora2}. (4) follows
from Lemma \ref{n res lemm}. Finally (5) follows from Lemmas \ref{n
res lemm} and  \ref{sora2}.
\end{proof}

\begin{remark}
Let $n=3,4$ and $(\lambda_c,\mu)\in {\mathfrak H}_0$. By (3) and (4)
of Lemma \ref{thres and res}  it can be seen that $0$ is  a
threshold resonance and a   threshold eigenvalue.
\end{remark}

\section{Spectrum  of $H_{\lambda}^{\mathrm{o}}$}
\subsection{Birman-Schwinger principle for $z\in {\mathbb C}\setminus[0,2n]$}
In the previous sections, we study the spectrum of
$H_{\lambda\mu}^{\mathrm{e}}$ by using the Birman-Schwinger
principle
 for $n\geq3$,
and by directly solving $H_{\lambda\mu}^{\mathrm{e}} f=0$ for
$n=1,2$. In the case of $H_{\lambda}^{\mathrm{o}}$  we can proceed
in a similar way to the the case of $H_{\lambda\mu}^{\mathrm{e}}$
and rather easier than that of $H_{\lambda\mu}^{\mathrm{e}}$ as is
seen below.
Let $z\in {\mathbb C}\setminus[0,2n]$. As is  done for
$H_{\lambda\mu}^{\mathrm{e}}$,  we  can see that
\begin{align*}
   (H_0-z)^{-1}V_{\lambda}^{\rm o}=S_1S_2.
   \end{align*}
Here $S_1$ and $S_2$ are defined by
\begin{align*}
&S_1:{\mathbb C}^n\ni \VVV{w_1\\ \vdots\\ w_n }\mapsto (H_0-z)^{-1}
\frac{\lambda }{2}\sum_{j=1}^n w_j
s_j\in  L^2_{\rm o}(\mathbb{T}^n),\\
&S_2:L^2_{\rm o}(\mathbb{T}^n)\ni \phi\mapsto
\VVV{\langle \phi,s_1\rangle\\
\vdots\\
\langle \phi,s_n\rangle }\in \mathbb{C}^n.
\end{align*}
We set
$$
G_{\rm o}(z)=S_2S_1:\mathbb{C}^n\to \mathbb{C}^n.
$$
The following assertion is  proved as Lemma \ref{lem 1e}, and then
we  omit the  proof.
\begin{lemma}[{Birman-Schwinger principle for $z\in {\mathbb C}\setminus[0,2n]$}]
\label{lem 1 o} \
\begin{itemize}
\item[(a)] $z\in \mathbb{C}\setminus
[0, 2n]$ is an eigenvalue of $H_{\lambda}^{\rm o}$ if and only if
$1\in \sigma(G_{\rm o}(z))$. 
\item[(b)] Let $z\in \mathbb{C}\setminus [0,2n]$ and
$Z=\VVV{ w_0\\  \vdots\\w_n}\in \mathbb{C}^{n}$
 be such that $G_{\rm o}(z)Z=Z$. Then
 $ f=S_1Z, $
\begin{align*}
f(p)= \frac{1}{(2\pi) }\frac{1}{E(p)-z}\left ( \frac{\lambda}{\sqrt
{2}} \sum_{j=1}^n w_j\sin p_j \right )
\end{align*}
 is an eigenfunction of $H_{\lambda}^{\mathrm{o}}$, i.e.,
 $H_{\lambda}^{\mathrm{o}} f=zf$.
\end{itemize}
 \end{lemma}
We see that $ \frac{1}{2}\langle \mathrm{s}_i,(H_0-z)^{-1}
\mathrm{s}_j \rangle =\frac{1}{(2\pi)^{n}} \int_{\mathbb{T}^n}
\frac{\sin p_i\sin p_j
 }{E(p)-z}dp=0$
by the fact that $E(p)=E(p_1,\dots,p_d)$ is even for any $p_j$.
Therefore
\begin{align}
G_{\rm o}(z)=\lambda s(z) I,
\end{align}
where $\displaystyle
s(z)=(2\pi)^{-n}\int_{\mathbb{T}^n}\frac{\sin^2 p_1}{E(p)-z}dp $ is
given by \eqref{sz}. Consequently we have for $n\geq 1$,
$$\delta_s(\lambda;z)=\det(G_{\rm o}(z)-{\rm I} )= (\lambda s(z)-1)^n.$$
Since $G_{\rm o}(z)$ is diagonal, it is very easy to find solution
of  $G_{\rm o}(z)Z=Z$. It has $n$ independent solutions:
$$
Z_j=\VVV{0\\ \vdots\\ 1\\ \vdots\\0} \! \leftarrow j_{th},\quad
j=1,...,n.$$ The corresponding eigenvector,
$H_{\lambda}^{\mathrm{o}} f_j=z f_j$,  is given by
\begin{align}\label{eigen Sin}
    f_j(p)=
    \frac{1}{E(p)-z}\frac{1}{(2\pi) } \frac{\lambda}{\sqrt 2} \sin p_j,\quad j=1,\dots,n,
\end{align}
where $\lambda=1/s(z)$. In particular the multiplicity of $z$ is
$n$.
\subsection{Birman-Schwinger principle for $z=0$}
We can extend the Birman-Schwinger principle for $z=0$. We extend
the eigenvalue equation $H_{\lambda}^{\mathrm{o}} f=0$ in $L_{\rm
o}^2(\mathbb{T}^n)$ to that in $L_{\rm o}^1(\mathbb{T}^n)$. We
consider the equation
\begin{equation}\label{Hf=0o}
 E(p) f(p)-\frac{\lambda}{(2\pi)^n}
 \sum_{j=1}^n
 \sin p_j \int_{\mathbb{T}^n}\sin p_j f(p)dp =0
\end{equation}
in  $L_{\rm o}^1(\mathbb{T}^n)$. We also describe \eqref{Hf=0o} as
$H_{\lambda}^{\mathrm{o}} f=0$. We can see that $\sin
p_j/E(p)\approx 1/|p|$ in the neighborhood of $p=0$, and then $\sin
p_j/E(p)\in L^1(\mathbb{T}^n)$ for $n\geq 2$. By (e) of Lemma
\ref{Incl L2} and $V_{\lambda}^{\rm o}f\in C(\mathbb{T}^n)$ we can
see that
\begin{align}
L^2_{\rm o}(\mathbb{T}^n)\ni f\mapsto     H_0^{-1}V_{\lambda}^{\rm
e}f \in
L^2_{\rm o}(\mathbb{T}^n),&\quad  n\geq 3,\\
L^1_{\rm o}(\mathbb{T}^n)\ni f\mapsto   H_0^{-1}V_{\lambda}^{\rm o}f
\in L^1_{\rm o}(\mathbb{T}^n) ,& \quad n\geq 2.
\end{align}
Thus for $n\geq 2$ we can extend operators $S_1$ and $S_2$. Let
$n\geq 2$ and $Z=\VVV{ w_0\\  \vdots\\w_n}$. $\bar S_1:
\mathbb{C}^{n} \to L^1_{\rm o}(\mathbb{T} ^n)$ is  defined by
$$\bar S_1 Z=
\frac{1}{(2\pi) } \frac{\lambda}{\sqrt {2}} \frac{1}{E(p)}
 \sum_{j=1}^n w_j\sin p_j $$
and $\bar S_2:L_{\rm o}^1(\mathbb{T} ^n)\to \mathbb{C}^{n}$  by
$$\bar S_2:L_{\rm o}^1(\mathbb{T} ^n)\ni \phi\mapsto \VVV{
\int_{T^n} \phi(p) s_1(p) dp\\
\vdots\\
\int_{T^n} \phi(p) s_n(p) dp }\in \mathbb{C}^{n}.$$ Then $
\overline{S}_1\overline{S}_2: L^1_{\rm o}(\mathbb{T} ^n)\to L^1_{\rm
o}(\mathbb{T} ^n) $. Thus $ G_{\rm
o}(0)=\overline{S}_2\overline{S}_1:\mathbb{C}^{n}\to \mathbb{C}^{n}
$ is described as an  $ n \times n $ matrix. Let $n\geq 2$. We have
(1) $\lim_{z\to 0} G_{\rm o}(z)=G_{\rm o}(0)$, and (2)
$\sigma(H_0^{-1}V_{\lambda}^{\rm o}) \setminus\{0\}=\sigma(G_{\rm
o}(0))\setminus\{0\} $. Hence for $n\geq2$,
\begin{align}
\label{tama1} G_{\rm o}(0)=\lambda s(0) I
\end{align}
and $\bar \delta_s(\lambda;z)$ is defined by
\begin{align}\label{tama2}
\bar \delta_s(\lambda;z)=\left\{
\begin{array}{ll}
\delta_s(\lambda;z)& z\in (-\infty,0),\\
(\lambda s(0)-1)^n &z=0.
\end{array}
\right.
\end{align}
\begin{remark}
In \eqref{tama1} and \eqref{tama2} we define $\bar
\delta_s(\lambda,z)$ and $G_{\rm o}$ for $n\geq2$. We note however
that $s(0)<\infty$ for $n\geq 1$. Thus $G_{\rm o}$ and $\bar
\delta_s(\lambda;z)$ are  well defined for  $n\geq 1$.
\end{remark}

\begin{lemma}[Birman-Schwinger principle for $z=0$]
\label{lem 1o res} Let $n\geq 2$.
\begin{itemize}
\item[(a)]
Equation $H_{\lambda}^{\mathrm{o}} f=0$ has a solution in
$L^1(\mathbb{T}^n)$ if and only if $1\in \sigma(G_{\rm o}(0))$.
\item[(b)]
 Let $Z=\VVV{ w_0\\  \vdots\\w_n}\in {\mathbb C}^n$ be the solution of $G_{\rm o}(0) Z=Z$ if and only if
$$f(p)=\bar S_1Z(p)=\frac{1}{(2\pi)^n} \frac{1}{E(p)}\frac{\lambda}{\sqrt 2} \sum_{j=1}^n w_j \sin p_j$$
is a solution of $H_{\lambda}^{\mathrm{o}} f=0$, where $w_1,\cdots,
w_n$ are actually described by
\begin{align}\label{wo}
w_j=\frac{\sqrt 2}{(2\pi)^{\frac{n}{2}}} \int_{\mathbb{T}^n}f(p)\sin
p_jdp,\quad j=1,\dots,n.
\end{align}
\end{itemize}
\end{lemma}
\begin{proof}
The proof is the same as that of Lemma \ref{lem 1e res}.
\end{proof}
\subsection{Eigenvalues of $H_\lambda^{\rm o}$}
Set
\begin{align*}
\lambda_s=\frac{1}{s(0)}.
\end{align*}
Note that $\lambda_s=1$ for $n=1$. We divide $(\lambda,\mu)$-plane
into two half planes $S_\pm$ and the boundary ${\mathfrak S}_0$. Set
$${\mathfrak S}_-=\{(\lambda,\mu)\in{\mathbb R}^2; \lambda<\lambda_s\},
{\mathfrak S}_0=\{(\lambda,\mu)\in {\mathbb R}^2;
\lambda=\lambda_s\}, {\mathfrak S}_+=\{(\lambda,\mu)\in{\mathbb
R}^2; \lambda>\lambda_s\}.$$ See Figure \ref{f82}.
\begin{figure}[h]
\begin{center}
\begin{tikzpicture}[scale=0.8]
\useasboundingbox(-1,-2) rectangle(4,4);
\draw[thick](-1,0)--(4,0)node[below]{$\lambda$};
\draw[thick](0,-1)--(0,4)node[left]{$\mu$};
\node[below]at(1.5,-1.5){$ $}; \draw[thin](-1,0)--(4,0);
\draw[thin](0,-1)--(0,4); \draw[very
thick](2,-1)node[below]{$\lambda_s$}--(2,4);
\node[scale=0.8]at(1,2){${ {\mathfrak S}_-}$};
\node[scale=0.8]at(3.9,2){${{\mathfrak S}_+}$};
\node[scale=0.8]at(2,4.3){${ {\mathfrak S}_0}$};
\end{tikzpicture}
\end{center}
\vspace{-0.5cm} \caption{Regions of $S_\pm$  for $n\geq 1$}
\label{f82}
\end{figure}
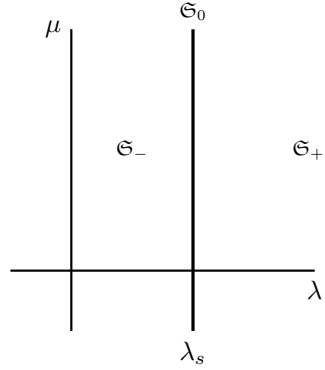

\begin{lemma}\label{Lemm delta sin }
Let $n\geq1$. Then (a)-(c) follow:
\begin{itemize}
\item[(a)] Let $(\lambda,\mu)\in {\mathfrak S}_-\cup {\mathfrak S}_0$. Then
$\bar \delta_s(\lambda;\cdot)$ has no zero in $(-\infty,0)$.
\item[(b)] Let $(\lambda,\mu)\in {\mathfrak S}_0$. Then $\bar \delta_s(\lambda_s;0)=0$
and
 $z=0$ has multiplicity $n$.
\item[(c)] Let $(\lambda,\mu)\in {\mathfrak S}_+$. Then
$\bar \delta_s(\lambda;\cdot)$ has a unique zero in $(-\infty,0)$
with multiplicity  $n$.
\end{itemize}
\end{lemma}
\begin{proof}  The proof is similar to that of
Lemma \ref{delta cos}, and hence  we omit it.
\end{proof}
By Lemma \ref{Lemm delta sin } we can see spectral property of
$H_{\lambda}^{\mathrm{o}}$.
\begin{lemma}[Eigenvalues of $H_{\lambda}^{\mathrm{o}}$]
Let $n\geq1$.
\begin{itemize}
\item[(1)] $(\lambda,\mu)\in {\mathfrak S}_-\cup {\mathfrak S}_0$.
There is no eigenvalue in $(-\infty,0)$.
\item[(2)] $(\lambda,\mu)\in {\mathfrak S}_0$.
There is an $n$ fold threshold eigenvalue  or threshold resonance.
\item[(3)] $(\lambda,\mu)\in {\mathfrak S}_+$.
There is an $n$ fold eigenvalue in $(-\infty,0)$.
\end{itemize}
\end{lemma}
\begin{proof}
This lemma follows from Lemmas \ref{Lemm delta sin }, and the fact
that $z\not=0$ is an eigenvalue if and only if $\bar
\delta_s(\lambda;z)=0$, and $0$ is an threshold eigenvalue  or
threshold resonance   if and only if $\bar \delta_s(\lambda;0)=0$.
\end{proof}

\subsection{Threshold eigenvalues and threshold resonances for $H_{\lambda}^{\mathrm{o}}$}
\label{subsec reson} Threshold resonances and threshold eigenvalues
for $H_{\lambda}^{\mathrm{o}}$ can be discussed by the
Birman-Schwinger principle for $n\geq 2$.
\begin{lemma}\label{reson lem odd}
Let $n\geq 2$. Then the solutions of equation
$H_{\lambda}^{\mathrm{o}} f=0$ are given by
\begin{align}\label{sake}
f_j(p)= \frac{1}{(2\pi) } \frac{\lambda_s}{\sqrt{2}} \frac{\sin
p_j}{E(p)},\quad j=1,\dots,n.
\end{align}
\end{lemma}
\begin{proof}
From $\bar \delta(\lambda_s,0)=0$ and Lemma \ref{lem 1o res} the
lemma follows.
\end{proof}
For $n=1$ we can directly see that $H_{\lambda}^{\mathrm{o}} f=0$
has no solution in $L^1$, but it has a   super-threshold resonance.
We see this in the next proposition.
\begin{proposition}[Super-threshold resonance]
\label{takahashi} Let $n=1$ and $\lambda=\lambda_s=1$. Then
$H_{\lambda}^{\mathrm{o}} f=0$ has solution $f\in L_{\rm
o}^\epsilon(\mathbb{T})\setminus  L_{\rm o}^1(\mathbb{T})$ for any
$0< \epsilon <1$. I.e., $0$ is a    super-threshold resonance of
$H_{\lambda}^{\mathrm{o}}$.
\end{proposition}
\begin{proof}
$H_{\lambda_s}^o f=0$ yields that $f(p)=C \frac{\sin p}{E(p)}$,
where $C=\frac{\lambda_s }{2\pi}\int_\mathbb{T} \sin p f(p) dp.$
Note that however $\sin p/E(p) \not\in L^1(\mathbb{T})$, but we can
see that $\sin p/E(p)\in L^\epsilon(\mathbb{T})$ for any
$0<\epsilon<1$ since $\sin p/E(p)\sim 1/p$ near $p=0$ and
$\int_\mathbb{T} p^{-\epsilon}dp<\infty$.
\end{proof}

\begin{lemma}\label{thres and res o}
\begin{itemize}
\item[(1)]
Let $n=1$. Then $0$ is neither a   threshold resonance nor a
threshold eigenvalue, {but for $(\lambda_s,\mu)$, $0$ is a
super-threshold resonance.}
\item[(2)]
Let $n=2$. Then $0$ is  a   threshold resonance at
$\lambda=\lambda_s$.
\item[(3)]
Let $n\geq 3$. Then $0$ is a   threshold eigenvalue at
$\lambda=\lambda_s$ and its multiplicity is $n$.
\end{itemize}
\end{lemma}
\begin{proof}
(1) follows from Proposition  \ref{takahashi}. Let $n\geq 2$. Then
the solution of $H_{\lambda}^{\mathrm{o}} f=0$ is given by
\eqref{sake}. Since
\begin{align*}
&\frac{\sin p_j}{E(p)}\in L^1(\mathbb{T}^n)\setminus L^2(\mathbb{T}^n),\quad  n=2,\\
&\frac{\sin p_j}{E(p)}\in L^2(\mathbb{T}^n),\quad  n\geq3,
\end{align*}
we have $f\in L^1(\mathbb{T}^n)\setminus L^2(\mathbb{T}^n)$ for
$n=2$, and $f\in L^2(\mathbb{T}^n)$ for $n\geq3$. Then (2) and (3)
follow.
\end{proof}

\section{Main theorems}

\subsection{Case of $n\geq2$}
In order to describe the main results we have to separate
$(\lambda,\mu)$-plane into several regions.
\begin{lemma}
Let $n\geq 2$. Then $\lambda_\infty(z)\leq \lambda_s(z)\leq
\lambda_c(z)$ for $z\in (-\infty,0]$.
\end{lemma}
\begin{proof}
By Lemma \ref{as<b} it follows that
$$\lambda_c(z)=\frac{1}{c(z)-d(z)}>\lambda_s(z)=\frac{1}{s(z)}>\lambda_\infty(z)=\frac{a(z)}{b(z)}$$ for $z<0$.
By a limiting argument the lemma follows.
\end{proof}
We introduce $4$-half planes:
\begin{align*}
{\mathfrak C}_-&=\{(\lambda,\mu)\in{\mathbb R}^2; \lambda< \lambda_c
\}, \quad {\mathfrak C}_+=\{(\lambda,\mu)\in{\mathbb R}^2; \lambda>
\lambda_c \}\\
{\mathfrak S}_-&=\{(\lambda,\mu)\in{\mathbb R}^2; \lambda< \lambda_s
\}, \quad {\mathfrak S}_+=\{(\lambda,\mu)\in{\mathbb R}^2; \lambda>
\lambda_s \},
\end{align*}
and two boundaries: ${\mathfrak C}_0 =\{(\lambda,\mu)\in{\mathbb
R}^2; \lambda= \lambda_c \}$ and ${\mathfrak
S}_0=\{(\lambda,\mu)\in{\mathbb R}^2; \lambda= \lambda_s \}$. Note
that ${\mathfrak S}_-\subset {\mathfrak C}_-$ and ${\mathfrak
C}_+\subset {\mathfrak S}_+$, and we define open sets surrounded by
hyperbola ${\mathfrak{H}_0}$ and boundary $\Gamma_c$ and $\Gamma_s$
by
:
\begin{align*}
    &D_0=G_0,
    \quad
      D_1=G_1 \cap {\mathfrak S}_-,\quad       D_2=G_2 \cap {\mathfrak S}_-,\quad D_{n+1}=G_1 \cap ({\mathfrak S}_+\cap  {\mathfrak C}_-),\\
 &     D_{n+2}=G_2 \cap ({\mathfrak S}_+\cap  {\mathfrak C}_-),\quad D_{2n}=G_1 \cap   {\mathfrak C}_+,\quad
     D_{2n+1}=G_2 \cap   {\mathfrak C}_+.
\end{align*}
The boundaries of these sets define disjoint 8 curves:
\begin{align*}
&B_0=\Gamma_l,\quad B_1=\Gamma_r\cap {\mathfrak S}_-,\quad
B_{n+1}=\Gamma_r\cap ({\mathfrak S}_+\cap {\mathfrak C}_-), \quad
B_{2n}=\Gamma_r\cap  {\mathfrak C}_+,
\quad\\
&S_{1}={\mathfrak S}_0 \cap G_1, \quad S_{2}={\mathfrak S}_0 \cap
G_2,\quad C_{n+1}={\mathfrak C}_0  \cap G_1, \quad
C_{n+2}={\mathfrak C}_0  \cap G_2,
\end{align*}
and two one point sets given by
$$
A=\Gamma_r\cap {\mathfrak S}_0,\quad B=\Gamma_r\cap {\mathfrak C}_0
.$$

We are now in the position to state the main theorem for $n\geq2$.
\begin{theorem}\label{Main 1} Let $n\geq 2$.
 \begin{itemize}
  \item
  [(a)] Assume  that $(\lambda,\mu)\in D_k,$ $k\in \{0,1,2,n+1,n+2,2n,2n+1\}$,
   then $H_{\lambda\mu}$ has
 $k$ eigenvalues in $(-\infty,0)$.
In addition $H_{\lambda\mu}$  has neither
 a    threshold eigenvalue nor a   threshold resonance
 (see Table~\ref{tab-1}).
\begin{table}
{
\begin{tabular}
  {|p{2.5cm}|p{1cm}|p{1cm}|p{1cm}|p{1cm}|p{1cm}|p{1cm}|p{1cm}|}
    \hline
  &  $D_0$
  &  $D_1$
  &  $D_2$
  &  $D_{n+1}$
  &  $D_{n+2}$
  &  $D_{2n}$
  &  $D_{2n+1}$
   \\
  \hline
 E.v.in $ (-\infty,0)$
   &  $0$
  &  $1$
  &  $2$
  &  ${n+1}$
  &  ${n+2}$
  &  ${2n}$
  &  ${2n+1}$
\\
\hline
\end{tabular}
} \caption{Spectrum of $H_{\lambda\mu}$ for $(\lambda,\mu)$ on $D_k$
for $n\geq2$.} \label{tab-1}
\end{table}

\item[(b)] $0$ is not a   super-threshold resonance of  $H_{\lambda\mu}$ for any $(\lambda,\mu)\in\mathbb{R}^2$.
\item[(c)] Assume that $(\lambda,\mu)$ in $B_k$, $S_k$, $C_k$ and $A,B$
 the next results are true in Table \ref{tab-2}:
\begin{table}
{
\begin{tabular}{|p{1.4cm}|p{1.9cm}|p{1.7cm}|p{1.9cm}|p{2cm}|p{2.2cm}|}
    \hline
  &  Curve $B_k$
 &  Curve $S_k$
 &
Curve $C_k$
 &
Point $A$ & Point $B$
\\
  \hline
 E.v.
  $ (-\infty,0)$
 &
$k$ &
 $k$ & 
$k $  & 
$1$ & 
$n+1$
\\
  \hline
Th.res.0
 &
\begin{tabular}{ll}
  $n=2$ & $-$\\
 \hline
 $n=3,4$ & $1$\\
 \hline
 $n\geq 5$ & $-$\\
\end{tabular}
&
\begin{tabular}{ll}
  $n=2$ & $2$\\
 \hline
 $n\geq 3$ & $-$\\
 \end{tabular}
  & 
\begin{tabular}{ll}
  $n\geq 2$ & $-$
\end{tabular}
 &
\begin{tabular}{ll}
  $n=2$ & $2$\\
 \hline
 $n=3,4$ & $1$\\
 \hline
 $n\geq 5$ & $-$
\end{tabular}
&
\begin{tabular}{ll}
  $n=2$ & $-$\\
 \hline
 $n=3,4$ & $1$\\
 \hline
 $n\geq 5$ & $-$\\
\end{tabular}
 \\
  \hline
Th.e.v.0
 &
\begin{tabular}{ll}
  $n=2$ & $-$\\
 \hline
 $n=3,4$ & $-$\\
 \hline
 $n\geq 5$ & $1$\\
\end{tabular}
&
\begin{tabular}{ll}
  $n=2$ & $-$\\
 \hline
 $n\geq 3$ & $n$\\
 \end{tabular}
 & 
\begin{tabular}{ll}
  $n\geq 2$ & $n-1$\\
 \end{tabular}
 &
\begin{tabular}{ll}
  $n=2$ & $-$\\
 \hline
 $n=3,4$ & $n$\\
 \hline
 $n\geq 5$ & ${\!\!n+1}$\\
\end{tabular}
&
\begin{tabular}{ll}
  $n=2$ & $1$\\
 \hline
 $n=3,4$ & $n-1$\\
 \hline
 $n\geq 5$ & $n$\\
\end{tabular}
 \\
 \hline
\end{tabular}
} \caption{Spectrum
 of $H_{\lambda\mu}$ for $(\lambda,\mu)$ on the edges of $D_k$ for $n\geq2$.}
\label{tab-2}
\end{table}
\end{itemize}
\end{theorem}
\begin{proof}
(a)  follows from 
Lemmas \ref{lemma Delta 0} and \ref{Lemm delta sin }. (b) follows
from Lemmas \ref{n=1 res lemm}, \ref{reson lem odd} and \ref{thres
and res o}. (c)  follows from Lemmas \ref{thres and res} and~
\ref{thres and res o}.
\end{proof}
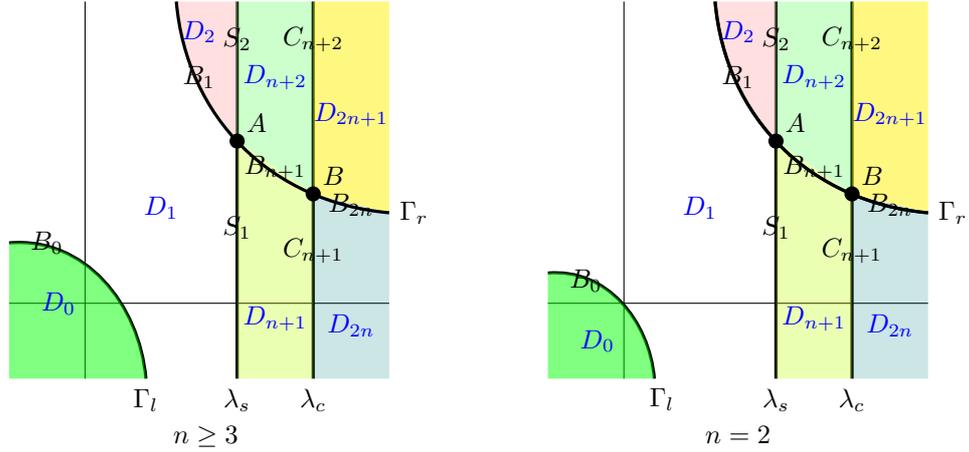
\begin{figure}[h]
\begin{center}
\begin{tikzpicture}[scale=1]
\useasboundingbox(-1,-2) rectangle(4,4);
\node[below]at(1.5,-1.5){$ $};
\draw[thin](-1,0)--(4,0); \draw[thin](0,-1)--(0,4);

\draw[very thick](2,-1)node[below]{$\lambda_s$}--(2,4); \draw[very
thick](3,-1)node[below]{$\lambda_c$}--(3,4);

\draw[very
thick](-1,0.8)to[out=50,in=140,relative](0.8,-1)node[below]{$\Gamma_l$};
\fill[green,opacity=.5](-1,0.8)to[out=50,in=140,relative](0.8,-1)--(-1,-1)--(-1,0.8);
\node[blue,left]at(0,0){$D_0$}; \node at(-0.5,0.8){$B_0$};

\fill[pink,opacity=.5](1.2,4)--(1.4,3)--(2,2.15)--(2,4);
\fill[green,opacity=.2](2,2.15)--(3,1.45)--(3,4)--(2,4);
\fill[yellow,opacity=.5](3,1.45)--(4,1.2)--(4,4)--(3,4);
\fill[lime,opacity=.3](2,-1)--(3,-1)--(3,1.45)--(2,2.15);
\fill[teal,opacity=.2](3,-1)--(4,-1)--(4,1.2)--(3,1.45);
\fill(2,2.15)circle[radius=0.1];
\fill(3,1.45)circle[radius=0.1];
\draw[very thick](1.2,4)to[out=320,in=220,relative](4,1.2)node[right]{$\Gamma_r$};

\node[above right]at(2,2.15){$A$}; \node[above
right]at(3,1.45){$B$}; \node[blue,above]at(1,1){$D_1$};
\node[blue]at(1.5,3.6){$D_2$}; \node[blue]at(2.5,3){$D_{n+2}$};
\node[blue]at(3.5,2.5){$D_{2n+1}$};
\node[blue]at(2.5,-0.2){$D_{n+1}$};
\node[blue]at(3.5,-0.3){$D_{2n}$}; \node at(2,1){$S_1$}; \node
at(2,3.5){$S_2$}; \node at(3,0.7){$C_{n+1}$}; \node
at(3,3.5){$C_{n+2}$}; \node at(1.5,3){$B_1$}; \node
at(2.5,1.8){$B_{n+1}$}; \node at(3.5,1.3){$B_{2n}$};
\end{tikzpicture}\hspace{2cm}
\begin{tikzpicture}[scale=1]
\useasboundingbox(-1,-2) rectangle(4,4);
\node[below]at(0.5,-1){$\Gamma_l$};
\draw[thin](-1,0)--(4,0); \draw[thin](0,-1)--(0,4);

\draw[very thick](2,-1)node[below]{$\lambda_s$}--(2,4); \draw[very
thick](3,-1)node[below]{$\lambda_c$}--(3,4);

\draw[very thick](-1,0.4)to[out=50,in=140,relative](0.4,-1);
\draw[very thick](1.2,4)to[out=320,in=220,relative](4,1.2);

\fill[green,opacity=.5](-1,0.4)to[out=50,in=140,relative](0.4,-1)--(-1,-1)--(-1,0.4);
\node[blue,left]at(0,-0.5){$D_0$}; \node at(-0.5,0.3){$B_0$};

\fill[pink,opacity=.5](1.2,4)--(1.4,3)--(2,2.15)--(2,4);
\fill[green,opacity=.2](2,2.15)--(3,1.45)--(3,4)--(2,4);
\fill[yellow,opacity=.5](3,1.45)--(4,1.2)--(4,4)--(3,4);
\fill[lime,opacity=.3](2,-1)--(3,-1)--(3,1.45)--(2,2.15);
\fill[teal,opacity=.2](3,-1)--(4,-1)--(4,1.2)--(3,1.45);
\fill(2,2.15)circle[radius=0.1];
\fill(3,1.45)circle[radius=0.1];
\draw[very thick](1.2,4)to[out=320,in=220,relative](4,1.2)node[right]{$\Gamma_r$};

\node[above right]at(2,2.15){$A$}; \node[above
right]at(3,1.45){$B$}; \node[blue,above]at(1,1){$D_1$};
\node[blue]at(1.5,3.6){$D_2$}; \node[blue]at(2.5,3){$D_{n+2}$};
\node[blue]at(3.5,2.5){$D_{2n+1}$};
\node[blue]at(2.5,-0.2){$D_{n+1}$};
\node[blue]at(3.5,-0.3){$D_{2n}$}; \node at(2,1){$S_1$}; \node
at(2,3.5){$S_2$}; \node at(3,0.7){$C_{n+1}$}; \node
at(3,3.5){$C_{n+2}$}; \node at(1.5,3){$B_1$}; \node
at(2.5,1.8){$B_{n+1}$}; \node at(3.5,1.3){$B_{2n}$};
\node[below]at(1.5,-1.5){$n=2$};
\node[below]at(-5.5,-1.5){$n\geq3$};
\end{tikzpicture}
\end{center}
\vspace{-0.5cm}\caption{Hyperbola for $n\geq 2$} \label{f8}
\end{figure}


We draw the results for $n=2$ on $(\lambda,\mu)$-plane in the
right-hand side of Figure \ref{f8} and for $n\geq3$ in the left-hand
side of Figure \ref{f8}.

\subsection{Case of $n=1$}
Let $n=1$. In this case,  the asymptote of $\mathfrak H_0$ is
$(\lambda_\infty(0),\mu_\infty(0))=(1,1)$, and $\lambda_c$ is not
defined. We also see that $\lambda_s=1=\lambda_\infty(0)$. Then we
have 4 sets:
\begin{align*}
      D_0=G_0,\quad
      D_1=G_1 \cap {\mathfrak S}_-,\quad
      D_2=G_1 \cap {\mathfrak S}_+,\quad
      D_3=G_2.
\end{align*}
The boundaries of these sets define disjoint 3 curves:
\begin{align*}
&B_0=\Gamma_l,\quad B_2=\Gamma_r,\quad S_{1} ={\mathfrak S}_0.
\end{align*}
Finally we define point $C$ by $C=\Gamma_r\cap {\mathfrak S}_0$. Now
we formulate next result for $n=1$.
\begin{theorem}\label{Main 2}
  Let $n=1$.
  \begin{itemize}
\item[(a)]   Assume  $(\lambda,\mu)\in D_k,$ $k\in \{0,1,2,3\}$.
  Then $H_{\lambda\mu}$ has
 $k$ eigenvalues in $(\infty,0)$. In addition
 $0$ is neither a   threshold resonance nor  a   threshold eigenvalue
 (see Table~\ref{tab-3}).
\begin{table}[h]
{ \begin{tabular}
  {|p{2.5cm}|p{1cm}|p{1cm}|p{1cm}|p{1cm}|}
    \hline
  &  $D_0$
  &  $D_1$
  &  $D_2$
  &  $D_3$
   \\
  \hline
 E.v.in $ (-\infty,0)$
  &  $0$
  &  $1$
  &  $2$
  &  $3$
\\
\hline
\end{tabular}}
\caption{Spectrum of $H_{\lambda\mu}$ for $(\lambda,\mu)$ on $D_k$
for $n=1$.}
\label{tab-3}
\end{table}
\item[(b)]
Assume that $(\lambda,\mu)\in S_1$. Then $H_{\lambda\mu}$ has a
super-threshold resonance.
\item[(c)]
Assume that $(\lambda,\mu)\in {B}_k\cup {S}_1$. Then  the next
result in Table \ref{tab-4} is true.
\begin{table}[h]
  {\begin{tabular}
  {|p{2.5cm}|p{2cm}|p{2cm}|}
    \hline
  &  $B_k$
 &  Curve $ S_k$\\
  \hline
 E.v.in $ (-\infty,0)$
 &
$k$ &
 $k$ \\
\hline
 Th.res.0
 &
$-$&
$-$\\
\hline Th.e.v.0
 &
$-$&
$-$\\
\hline
\end{tabular}}
\caption{Spectrum of $H_{\lambda\mu}$ for $(\lambda,\mu)$ on the
edges of $D_k$ for $n=1$.}
\label{tab-4}
\end{table}
 \end{itemize}
In particular $H_{\lambda\mu}$ has neither a   threshold resonance
nor a   threshold eigenvalue.
\end{theorem}
\begin{proof}
The theorem follows from 
Lemmas \ref{toto}, \ref{thres and res},  \ref{Lemm delta sin } and
\ref{thres and res o}.
\end{proof}
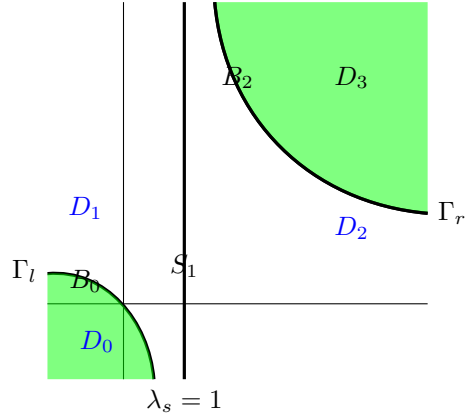
\begin{figure}[h]
\begin{center}
\begin{tikzpicture}[scale=1]
\useasboundingbox(-1,-2) rectangle(4,4);
\node[below]at(-1.3,0.7){$\Gamma_l $}; \draw[thin](-1,0)--(4,0);
\draw[thin](0,-1)--(0,4); \draw[very
thick](0.8,-1)node[below]{$\lambda_s=1$}--(0.8,4);
\draw[very thick](-1,0.4)to[out=50,in=140,relative](0.4,-1);
\draw[very thick](1.2,4)to[out=320,in=220,relative](4,1.2);
\fill[green,opacity=.5](1.2,4)to[out=320,in=220,relative](4,1.2)--(4,1.2)--(4,4);
\fill[green,opacity=.5](-1,0.4)to[out=50,in=140,relative](0.4,-1)--(-1,-1)--(-1,0.4);
\node[blue,left]at(0,-0.5){$D_0$}; \node at(-0.5,0.3){$B_0$};
\draw[very thick](1.2,4)to[out=320,in=220,relative](4,1.2)node[right]{$\Gamma_r$};
\node[blue,above]at(-0.5,1){$D_1$}; \node[blue]at(3,){$D_{2}$};
\node at(0.8,0.5){$S_1$}; \node at(3,3){$D_3$}; \node
at(1.5,3){$B_2$};
\end{tikzpicture}
\end{center}
\vspace{-1cm}\caption{Hyperbola for $n=1$} \label{f6}
\end{figure}

We draw the results for $n=1$ on $(\lambda,\mu)$-plane in Figure
\ref{f6}. In particular $(\lambda,\mu)\in S_1\cup S_2$

\subsection{Eigenvalues and asymptote}
From results obtained in the previous section a stable point on
$\lambda$ can be found. In general the spectrum of $H_{\lambda\mu}$
is changed  according to varying $\mu$ with a fixed $\lambda$. This
can be also seen from Figures \ref{f4}, \ref{f8} and \ref{f6}.
Curves on these figures  consist of only hyperbolas and vertical
lines. Then an asymptote has no intersection of these lines. It can
be seen that
$$\{(1,\mu)\in\mathbb{R}^2;\mu\in\mathbb{R}\}\cup\{(\lambda,n)\in\mathbb{R}^2;\lambda\in\mathbb{R}\}$$
is the asymptote of hyperbola $\mathcal{H}_z(\lambda,\mu)$ for
$n=1,2$. On the other hand
$$\{(a(0)/b(0),\mu)\in\mathbb{R}^2;\mu\in\mathbb{R}\}\cup\{(\lambda,n)\in\mathbb{R}^2;\lambda\in\mathbb{R}\}$$
is the asymptote of hyperbola $\mathcal{H}_z(\lambda,\mu)$ for
$n\geq 3$. Then we can have the corollary below.
\begin{corollary}
Let $\lambda=1$ for $n=1,2$ and $\lambda=a(0)/b(0)$ for $n\geq 3$.
\begin{itemize}
\item[(1)]
Let $n=1$. Then $H_{\lambda\mu}$ has a super-threshold resonance $0$
and only one eigenvalue in $(-\infty,0)$ for any $\mu$.
\item[(2)]
Let $n\geq 2$. Then $H_{\lambda\mu}$ has only one eigenvalue in
$(-\infty,0)$ for any $\mu$.
\end{itemize}
\end{corollary}
\begin{proof}
For $n=1,2$, let $l_n=S_1$. From Figures \ref{f8} and \ref{f6} it
follows that $l_n\cap \Gamma_l=l_n\cap \Gamma_r=\emptyset$. Then the
corollary follows. For $n\geq3$, let
$l_n=\{(a(0)/b(0),\mu)\in\mathbb{R}^2;\mu\in\mathbb{R}\}$. We can
also see that $l_n\cap \Gamma_l= l_n\cap S_1 = l_n\cap S_2 =l_n\cap
C_{n+1}= l_n\cap C_{n+2}= \emptyset$. Then the corollary follows.
\end{proof}

\appendix

\section{Proof of Lemma \ref{as<b}}\label{ineq lem}
\begin{proof}  We can see that 
   \begin{align}\label{s=}
        s(z)&=1+z(a(z)+b(z)),\quad n=1,\\
            s(z)&=1+z(a(z)+b(z))-(n-1)(a(z)-d(z)),\quad n \geq
            2,
\end{align}
and
$a(z)=\frac{1}{\sqrt{-z}\sqrt{2-z}}$ for $n=1$.\\
{(\bf Case $n=1$)} From $a(z)-b(z)=\frac{1}{n}+\frac{z}{n}a(z)$ we
see that $ b(z)=a(z)(1-z)-1 $.
Employing \eqref{s=} and $a(z)=\frac{1}{\sqrt{-z}\sqrt{2-z}}$,   we have 
\begin{align}%
a(z)s(z)=a(z)+z(a^2(z)+a^2(z)(1-z)-a(z))= a(z)(1-z)-1=b(z),\, z\leq
0.
\end{align}
{\bf (Case $n\geq 2$)}
 By 
\begin{align*}
   \frac{1}{2\pi}  \int_{\mathbb{T} } \frac{\cos
   p}{E(p)-z}dp =
     \frac{1}{(2\pi)^{2}}
 \left(     \int_{\mathbb{T} } \frac{\sin^2 p}{E(p)-z}dp
     \right)
      \left(     \int_{\mathbb{T} } \frac{1}{E(p)-z}dp
     \right),
\end{align*}
 we
obtain
\begin{align*}
   b(z)
   =   \frac{1}{(2\pi)^{n+1}}
\int_{\mathbb{T}^{n-1}}
 \left(     \int_{\mathbb{T} } \frac{\sin^2 p_1}{E(p)-z}dp _1
     \right)
      \left(     \int_{\mathbb{T} } \frac{1}{E(p)-z}dp_1
     \right)
   dp _2\dots dp _n,
\end{align*}
which provides
 \begin{align*}
   b(z)=\frac{1}{(2\pi)^{2n}} \int_{\mathbb{T}^{n-1}\times \mathbb{T}^{n-1}}
   F(\tilde{p})G(\tilde{q})d\tilde{p}d\tilde{q},
\end{align*}
where $\tilde{p}=(p_2,\dots,p_n)$, $\tilde{q}=(p_2,\dots,p_n)$, $
F(\tilde{p})=    \int_{\mathbb{T} } \frac{\sin^2
   p_1}{E(p_1,\tilde{p})-z}dp _1$ and
   $ G(\tilde{q})=    \int_{\mathbb{T} }
   \frac{1}{E(p_1,\tilde{q})-z}dp _1$.
Then
\begin{align*}
  a(z) s(z)=\frac{1}{(2\pi)^{2n}}\int_{\mathbb{T}^{n-1}\times \mathbb{T}^{n-1}}
   F(\tilde{p})G(\tilde{q})d\tilde{p}d\tilde{q},
\end{align*}
and we can have the relations: 
\begin{align}
 & a(z) s(z)-b(z)=-\frac{1}{2(2\pi)^{2n}} \int_{\mathbb{T}^{n-1}\times \mathbb{T}^{n-1}}
 \!\!\!
 \!\!\!
 \!\!\!
 \!\!\!
   \big(F(\tilde{p})-F(\tilde{q})\big)    \big(G(\tilde{p})-G(\tilde{q})\big) d\tilde{p}d\tilde{q},\nonumber \\
   \nonumber &
   \big(F(\tilde{p})-F(\tilde{q})\big)
     \big(G(\tilde{p})-G(\tilde{q})\big)\\
     &\label{(f-f)(g-g)}
=\left(\sum_{j=2}^{n}(\cos p_j-\cos q_j)\right)^2
 \!\!\!
\int_{\mathbb{T} }
\frac{\sin^2p_1dp_1}{(E(p_1,\tilde{p})-z)(E(p_1,\tilde{q})-z)}
 \!\!\!\int_{\mathbb{T} }\frac{dp_1}{(E(p_1,\tilde{p})-z)(E(p_1,\tilde{q})-z)}
   \end{align}
prove $ a(z) s(z)<b(z)$ for $n\geq 2$. Furthermore
\eqref{(f-f)(g-g)} shows that the last inequality leaves its sign
invariant even for $z=0$ and $n\geq 3$. Now we prove \eqref{s=1}. By
\begin{align*}
    &c(0)-d(0)= \frac{1}{2} \frac{1}{(2\pi)^{n}} \int_{\mathbb{T}^n} \frac{(\cos p_1
- \cos p_2)^2  }{E(p)}dp \\
    &s(0)= \frac{1}{2} \frac{1}{(2\pi)^{n}} \int_{\mathbb{T}^n} \frac{(\sin p_1
- \sin p_2)^2 }{E(p)} dp,
\end{align*}
we describe
\begin{align*}
    c(0)-d(0)= &\frac{1}{2} \frac{1}{(2\pi)^{n}}
    \int_{\mathbb{T}^n} \frac{4 \sin^2\frac{p_1-p_2}{2}
  \sin^2\frac{p_1+p_2}{2}
 }{E(p)}dp\\
    s(0)=&\frac{1}{2} \frac{1}{(2\pi)^{n}} \int_{\mathbb{T}^n} \frac{
4\sin^2\frac{p_1-p_2}{2}
  \cos^2\frac{p_1+p_2}{2}}{E(p)}dp.
\end{align*}
Introducing new variables $u=(p_1-p_2)/2$ and $t=(p_1+p_2)/2$ we get
\begin{align}\label{c-d-s}
    c(0)-d(0)-s(0)
=- 2 \frac{1}{(2\pi)^{n}} \int_{\mathbb{T}^{n-2}}dp_3\dots dp_{n},
\int_{\mathbb{T} }
 4
\sin^2 u  du \int_{\mathbb{T} }\frac{ \cos 2t}{ A-2\cos t \cos u
}dt,
\end{align}
where $ A=2+\sum_{j=3}^n(1-\cos p_j) $ is a function being
independent of both $t$ and $u$. We have
\begin{align*}
    &\int_{\mathbb{T} }\frac{ \cos 2t}{ A-2\cos t
\cos u }dt =
16 A \mathrm{c}\int_{0}^{\pi/4}
  \frac{ \cos^2 2t  \cos^2 u
 }{ (A^2-(2\cos t \cos u)^2)( A^2-(2\sin t \cos u)^2)
}
 dt >0.
\end{align*}
Using the last inequality  to \eqref{c-d-s} we get \eqref{s=1}.
\end{proof}

 \section{Proof of Lemma \ref{lem asymp}}
\label{app lem assymp}
\begin{proof}  First we prove
\begin{align}\label{ab  less ab deri}
a'(z)b(z)-a(z)b'(z)<0,\quad
 z\in
(-\infty,0),
\end{align}
which proves the monotone decreasing of $\frac{a(z)}{b(z)}$. The
equality
\begin{align*}
    a'(z)b(z)-a(z)b'(z)=    \frac{1}{(2\pi)^{2n}} \int_{\mathbb{T}^n\times \mathbb{T}^n}
\cos p_1 \frac{E(p)-E(q)}{(E(p)-z)^2 (E(q)-z)^2 } dp  dq
\end{align*}
gives 
\begin{align*}
    a'(z)b(z)-a(z)b'(z)=    \frac{1}{n(2\pi)^{2n}} \sum_{j=1}^n\int_{\mathbb{T}^n\times \mathbb{T}^n}
\cos p_j \frac{E(p)-E(q)}{(E(p)-z)^2 (E(q)-z)^2 } dp  dq.
\end{align*}
Changing  variables $ E(p)-E(q)=\sum_{j=1}^n (\cos p_j-\cos q_j) $
provides the inequality
\begin{align*}
    a'(z)b(z)-a(z)b'(z)=   - \frac{1}{n(2\pi)^{2n}} \int_{\mathbb{T}^n\times \mathbb{T}^n}
\frac{\big(\sum_{j=1}^n(\cos p_j-\cos q_j)\big)^2 }{(E(p)-z)^2
(E(q)-z)^2 } dp  dq <0
\end{align*}
which proves  $\left(\frac{a(z)}{b(z)}\right)^{'}<0$ in
$(-\infty,0)$. Using  the definition of $a(z)$, we achieve $
a(z)=O(\frac{1}{|z|})$ and $b(z)=O(\frac{1}{z^2})$ as $z\to -\infty
$, and hence $ {a(z)}/{b(z)}=O(|z|)$ as $z\to -\infty$ proves
\eqref{lim 1}. Let $n=1,2$. By virtue of Lemma \ref{alpha - b} we
may write
$$
\frac{a(z)}{b(z)}=\frac{1}{1-\frac{b(z)-a(z)}{a(z)}} =
\frac{1}{1-\frac{1}{a(z)n}-\frac{z}{n}},
$$
and since $a(z)=O(\frac{1}{z})$ as $z\to 0-$ we receive \eqref{lim
2}. In case $n\geq 3$, the limit \eqref{lim 2} is obvious.
\end{proof}

\section*{Acknowledgement(s)}

 FH is financially supported by Grant-in-Aid
for Science Research (B)16H03942 and Challenging Exploratory
Research 15K13445 from JSPS. We thank Kota Ujino for the careful
reading of the manuscript and Tomoko Eto for drawing Figures
\ref{f3}-\ref{f6}.

\end{document}